\documentclass{article}[12pt]
\usepackage{graphicx} 
\usepackage{amsmath}
\usepackage{amsthm}
\usepackage{amssymb}
\usepackage{mathtools}
\usepackage{graphicx}
\usepackage{xcolor}
\usepackage[hidelinks, bookmarks = false]{hyperref}
\usepackage{enumitem}
\usepackage{bbm}

\usepackage{caption}
\usepackage[square,sort,comma,numbers]{natbib} 
\usepackage{tikz} 

\usepackage{setspace}
\usepackage{cleveref}
\theoremstyle{plain}

\newtheorem{theorem}{Theorem}[section]
\crefname{theorem}{Theorem}{Theorems}
\Crefname{theorem}{Theorem}{Theorems}

\newtheorem*{lemma*}{Lemma}
\newtheorem{lemma}[theorem]{Lemma}
\crefname{lemma}{Lemma}{Lemmas}
\Crefname{lemma}{Lemma}{Lemmas}

\newtheorem*{claim*}{Claim}

\crefname{claim}{Claim}{Claims}
\Crefname{claim}{Claim}{Claims}

\newtheorem{proposition}[theorem]{Proposition}
\crefname{proposition}{Proposition}{Propositions}
\Crefname{proposition}{Proposition}{Propositions}

\newtheorem{corollary}[theorem]{Corollary}
\crefname{corollary}{Corollary}{Corollaries}
\Crefname{corollary}{Corollary}{Corollaries}

\crefname{conjecture}{Conjecture}{Conjectures}
\Crefname{conjecture}{Conjecture}{Conjectures}

\newtheorem{question}[theorem]{Question}
\crefname{question}{Question}{Questions}
\Crefname{question}{Question}{Questions}

\crefname{observation}{Observation}{Observations}
\Crefname{observation}{Observation}{Observations}

\crefname{example}{Example}{Examples}
\Crefname{example}{Example}{Examples}

\theoremstyle{definition}

\crefname{problem}{Problem}{Problems}
\Crefname{problem}{Problem}{Problems}

\newtheorem{definition}[theorem]{Definition}
\crefname{definition}{Definition}{Definitions}
\Crefname{definition}{Definition}{Definitions}

\usepackage{xpatch}
\xpatchcmd{\proof}{\itshape}{\normalfont\proofnamefont}{}{}
\newcommand{\proofnamefont}{}
\renewcommand{\proofnamefont}{\bfseries}

\usepackage{framed}

\newcommand{\remove}[1]{}
\newcommand{\ceil}[1]{
    \lceil #1 \rceil
}
\newcommand{\floor}[1]{
    \lfloor #1 \rfloor
}

\newcommand{\al}{\alpha}
\newcommand{\be}{\beta}

\DeclareMathOperator{\ex}{ex}

\newcommand{\C}{\mathcal{C}}

\newcommand{\HH}{\mathcal{H}}

\newcommand{\D}{\mathcal{D}}
\newcommand{\F}{\mathcal{F}}
\newcommand{\G}{\mathcal{G}}

\newcommand{\I}{\mathcal{I}}
\newcommand{\cL}{\mathcal{L}}
\newcommand{\Pa}{\mathcal{P}}

\newcommand{\T}{\mathcal{T}}
\newcommand{\RR}{\mathbb{R}}
\newcommand{\FF}{\mathbb{F}}
\newcommand{\EE}{\mathbb{E}}
\newcommand{\PP}{\mathbb{P}}

\providecommand{\keywords}[1]
{
  \textbf{\textit{Keywords:}} #1
}

\providecommand{\MSC}[1]
{
  \textbf{\textit{2020 MSC Codes:}} #1
}

\newcommand{\ve}{\varepsilon}
\title{Regularization and asymmetric extremal numbers of subdivisions}
\author{Tao Jiang \footnote{Dept. of Mathematics, Miami University, Oxford, OH 45056, USA, {\tt jiangt@miamioh.edu}  }\and Sean Longbrake \footnote{Dept. of Mathematics, Emory University,  Atlanta, GA 30322, USA {\tt sean.longbrake@emory.edu}} }

\begin{document}

\maketitle

\begin{abstract}
Given a real $\mu\geq 1$, a graph $H$ is {\it $\mu$-almost-regular} if $\Delta(H)\leq \mu \delta(H)$.
The celebrated regularization theorem of Erd\H{o}s and Simonovits \cite{ES} states that for every real $0<\ve<1$ there exists a real $\mu=\mu(\ve)$ such that every $n$-vertex graph $G$ with $\Omega(n^{1+\ve})$ edges contains an $m$-vertex $\mu$-almost-regular subgraph $H$ with $\Omega(m^{1+\ve})$ edges for some $n^{\ve\frac{1-\ve}{1+\ve}}\leq m\leq n$.
We develop an enhanced version of it in which the subgraph $H$ also has average degree at least 
$\Omega(\frac{d(G)}{\log n})$, where $d(G)$ is the average degree of $G$. We then give a bipartite analogue of the enhanced regularization theorem. 

Using the bipartite regularization theorem, we show that
the maximum number of edges in a bipartite graph
with part sizes $m$ and $n$ that does not contain a $2k$-subdivision of $K_{s,t}$  or a $2k$-multi-subdivision of $K_p$ is no more than $O(m^{\frac{1}{2}+\frac{1}{2k}-\frac{1}{2ks}} n^{\frac{1}{2}}+n\log m)$ and $O(m^{\frac{1}{2}+\frac{1}{2k}} n^{\frac{1}{2}}+n\log m)$, respectively. These results extend the corresponding work of Janzer in \cite{Janzer-bipartite}  and \cite{Janzer-longer} 
to the bipartite setting for even subdivisions.  
We also show these upper bounds are tight up to a constant factor for infinitely many pairs $(m,n)$. The problem for estimating the maximum number of edges in a bipartite graph with part sizes $m$ and $n$ that does not contain a $(2k+1)$-subdivision of $K_{s,t}$ remains open.

\end{abstract}


\MSC{05C35}

\section{Introduction}

Given a graph $H$, a graph $G$ is {\it $H$-free} if it does not contain $H$ as a subgraph. As a fundamental problem in extremal graph theory, the Tur\'an problem studies the {\it extremal number} $\ex(n,H)$, defined to be the maximum number of edges in an $n$-vertex $H$-free graph. 
While the problem is essentially solved for non-bipartite graphs due to the celebrated theorems
of Erd\H{o}s and Stone \cite{E-Stone} and Erd\H{o}s and Simonovits \cite{ES}, the problem is still generally open when $H$ is bipartite. Nevertheless, there have been many promising progresses on the Tur\'an problem for bipartite graphs in recent years. One particular tool that has played an instrumental role in some of these progresses is  the following regularization theorem of Erd\H{o}s and Simonovits \cite{ES-regular}.
Roughly speaking the theorem says that a reasonably dense graph always contains a subgraph of similar {\it relative density} with the added property that the maximum and minimum degrees are within a constant factor of each other. Most existing techniques on Tur\'an type extremal problems rely on effective counting of various substrutures. By reducing the host graph to an almost regular subgraph with similar relative density, the counting tasks become much more manageable. Indeed, this is an important first step used in some of the recent progresses on the Tur\'an problem for bipartite graphs
(see \cite{CJL, Janzer-longer, Janzer-regular, Jiang-Qiu2, Sudakov-Tomon} for some examples). 
As ususal, we use $\Delta(G), \delta(G)$ to denote the maximum and minimum degree of a graph $G$, respectively.

\begin{theorem}[Erd\H{o}s-Simonovits regularization theorem \cite{ES-regular}]
Let $\ve$ be a real satisfying $0<\ve<1$. Let $n$ be a sufficiently large positive integer in terms of $\ve$. Let $G$ be an $n$-vertex graph with $e(G)\geq n^{1+\ve}$. Then $G$ contains a subgraph $H$ on 
$m\geq n^{\ve\frac{1-\ve}{1+\ve}}$ vertices such that $e(H)\geq \frac{2}{5} m^{1+\ve}$ and
$\Delta(G')\leq c_\ve\delta(G')$, where $c_\ve= 20 \cdot 2^{\frac{1}{\ve^2}+1}$.
\end{theorem}

There have been several adaptions of the Erd\H{os}-Simonovits regularization theorem, see for instance
\cite{CJL, Jiang-Seiver}. Very recently, Chakraborti, Janzer, Methuku, and Montgomery \cite{CJMM} proved the following remarkable strengthening of the Erd\H{o}s-Simonovits regularization theorem.

\begin{theorem}[\cite{CJMM}]
For any positive real $\ve<1$, there exists some $\beta>0$ such that the following holds. Let $c>0$ and let $n$ be a positive integer that is sufficiently large as a function of $\ve$ and $c$. Let $G$ be a graph on $n$ vertices with $e(G)\geq cn^{1+\ve}$. Then $G$ contains a regular subgraph $H$ on $m\geq n^\beta$ vertices such that $e(H)\geq \beta c m^{1+\ve}$.    
\end{theorem}

Note that in both of these theorems, while the resulting (almost) regular subgraph $H$ preserves the same relative density as $G$, $v(H)$ is potentially very small compared to $v(G)$ and hence the average degree of $H$ could become quite small in the absolute sense. In some applications, it may be desirable to have an almost-regular subgraph of $G$ that not only preserves the relative density but also has average degree close to the original average degree.

We show this can be accomplished by proving the following theorem.  Given a real $\mu\geq 1$, we say that a graph $G$ is {\it $\mu$-almost-regular} if  $\Delta(G)\leq \mu \delta(G)$. We also use $d(G)$ to denote the average degree of $G$.

   \begin{theorem}[Enhanced regularization theorem]  \label{thm:enhanced-regularization}
       Let $c$ be any positive real. Let $G$ be an $n$-vertex graph with at least $cn^{1 + \ve}$ edges. Then, there is a $6$-almost-regular subgraph  $H$ on  some $m$ vertices such that
\[e(H)\geq \frac{2^\ve-1}{48}c m^{1+\ve}\, \mbox{\rm and } \,  d(H)\geq \frac{1}{12}\frac{d(G)}{\log (2n/d(G))}.\]    
   \end{theorem}

It follows from the proof of Proposition 5.2 of \cite{CJMM} that for any fixed $\mu, \ve, c$
there exists $c'>0$ such that for sufficiently large $n$ there exists an $n$-vertex graph $G$ with at least $cn^{1+\ve}$ edges and no $\mu$-almost-regular subgraph with average degree at least $c'\frac{d(G)}{\log n}$. (The same conclusion can be also be derived by combining Theorem 1.13 and Proposition 5.2 of \cite{CJMM}.)  So, the lower bound on $d(H)$ in Theorem \ref{thm:enhanced-regularization} is asymptotically best possible.

Next, we wish to develop analogues of regularization theorems in the setting where the host graph is bipartite with potentially drastically different part sizes.  Our main motivation here is that
a biregularization theorem that introduces some regularity to the host graph will aide the counting of relevant substructures, which can be useful in the study of extremal numbers in the bipartite setting.

Here, we define the asymmetric extremal number of $H$, denoted by $\ex(m,n,H)$, to be maximum number of edges in an $H$-free bipartite graph with part sizes $m,n$, respectively, where $m\leq n$. The study of $\ex(m,n,H)$ is even more complex than $\ex(n,H)$, as the behavior of the function is also affected by range of values of $m$ relative to $n$. Results on $\ex(m,n,H)$ have been relatively sporadic. See for instance \cite{ESS-product}, \cite{Gyori}, and \cite{NV} for the study of $\ex(m,n,C_{2k})$. We refer the reader to the survey of F\"uredi and Simonovits \cite{FS} for more discussion on $\ex(m,n,H)$ and the related Zarankewicz problem $z(m,n,H)$.

For the main purpose of this paper, we will use the following natural biregularity notion. In the concluding remarks, we will mention a weaker notion of biregularity that can be more useful in certain applications.

\begin{definition} [$\mu$-almost-biregularity]
    Given a positive real $\mu\geq 1$, a bipartite graph $G$ with a bipartition $(A,B)$ is called {\it $\mu$-almost-biregular} if $\Delta_A\leq \mu \delta_A$ and $\Delta_B\leq \mu \delta_B$ hold, where 
    $\Delta_A, \Delta_B$ denote the maximum degree of a vertex in $A$ and in $B$, respectively
    and $\delta_A, \delta_B$ denote the minimum degree of a vertex in $A$ and in $B$, respectively.
\end{definition}

\begin{theorem}[Biregularization theorem] \label{thm:biregular}
    Let $0<\alpha, \beta\leq 1$ be reals satisfying  $\alpha+\beta>1$.
    There exists a positive constant $\lambda=\lambda(\alpha,\beta)$
    such that the following holds. 
    Let $G$ be a bipartite graph with an ordered partition $(M, N)$, with $|M| = m$, $|N| = n$, and $m \leq n$ such that $e(G) \geq cm^{\alpha}n^{\beta}$ and $d(G)\geq 8$. Then $G$ contains a $16$-almost-biregular subgraph $G'$, with a partition $(M',N')$ where $M' \subseteq M, N' \subseteq N$, $|M'| = m', |N'| = n'$ such that  \[e(G') \geq  \lambda c (m')^{\alpha}(n')^{\beta}\, \mbox{\rm and }\,d(G') \geq \frac{1}{64}\frac{d(G)}{\log m}.\] 
\end{theorem}

Next, we apply our biregularization theorem to estimate the extremal number of subdivisions in the bipartite setting.   

 Given a graph $H$ and positive integers $k,r\geq 2$, we denote by $H^{(k)}$ the graph obtained from $H$ by replacing each edge $uv$ with a $u,v$-path $P_{u,v}$ of length $k$ so that the paths $P_{u,v}$'s are pairwise internally disjoint. We call $H^{(k)}$ a {\it simple $k$-subdivision} of $H$.
We denote by $H^{(k)}(r)$ the graph obtained from $H$ by replacing each edge $uv$ with $r$ internally
disjoint $u,v$-paths of length $k$ such that all these $re(G)$ paths are internally disjoint. We call $H^{(k)}(r)$, the {\it $r$-multi-subdivision} of $H$. The study of $\ex(n,H^{(k)})$ was initiated by Jiang and Seiver \cite{Jiang-Seiver}. Their bounds were significantly improved by Conlon, Janzer, and Lee \cite{CJL}. (see also followup work by Jiang and Qiu \cite{Jiang-Qiu}.) Janzer \cite{Janzer-bipartite, Janzer-longer} subsequently essentially settled the problems for $\ex(n, K_{s,t}^{(k)})$ and $\ex(n,K_p^{(2k)}(r))$. For $K_{s,t}^{(k)}$, it is shown \cite{Janzer-bipartite} that $\ex(n,K_{s,t}^{(k)})=O(n^{1+\frac{1}{k}-\frac{1}{sk}})$, which is asymptotically tight when $t$ is sufficiently large in terms of $s$. For $K_p^{(2k)}$, it is shown \cite{Janzer-longer} that $\ex(n, K_p^{(2k)}(r))=O(n^{1+\frac{1}{k}})$, which is asymptotically tight when $r$ is sufficiently large, and
that $\ex(n,K_p^{(2k)})=O(n^{1+\frac{1}{k}-\delta})$ for some $\delta>0$. 
We establish the bipartite analogues of the results on $\ex(n,K_{s,t}^{(k)})$ and $\ex(n,K_p^{(k)}(r))$ for even $k$ as follows.

\begin{theorem} \label{thm:main-kst}
    Let $k,r, s,t$ be positive integers. Then for all positive integers $m\leq n$,
    \[\ex(m,n, K_{s,t}^{(2k)}(r))\leq O(m^{\frac{1}{2}+\frac{1}{2k}}n^{\frac{1}{2}}+n\log m),\mbox{ and } \ex(m,n, K_{s,t}^{(2k)})\leq O(m^{\frac{1}{2}+\frac{1}{2k}-\frac{1}{2ks}} n^{\frac{1}{2}}+n\log m).\]
\end{theorem}

Furthermore, in Theorem \ref{thm:theta-lower-even}, by establishing a lower bound for so-called theta graphs (i.e. the union of $r$ internally disjoint paths of length $2k$ sharing the same endpoints), we will show that the first upper bound is asymptotically tight for infinitely many pairs when $r$ is sufficiently large. 
In Theorem \ref{thm:kst-lower-even}, we will show that the second upper bound is asymptotically tight for infinitely many pairs $m,n$ when $t$ is sufficiently large. The situation with odd subdivisions of $K_{s,t}$ is more difficult. We will raise the natural question in the concluding remarks.

\begin{theorem} \label{thm:main-kp}
    Let $k,p,r$ be positive integers. Then for all positive integers $m\leq n$,
    \[\ex(m,n, K_{p}^{(2k)}(r))\leq O(m^{\frac{1}{2}+\frac{1}{2k}} n^{\frac{1}{2}}+n\log m).\]
\end{theorem}
As commented above, the  upper bound is asymptotically tight for infinitely many pairs when $r$ is sufficiently large. Let us remark that while our general strategy for both theorems follows that used by Janzer in \cite{Janzer-bipartite} and \cite{Janzer-longer}, one main obstacle to proving Theorem \ref{thm:main-kp} lies in the very delicate role paths of length two play due to the unbalanced part sizes of the host graph. The main novelty of the proof of Theorem \ref{thm:main-kp} lies in our handling of the paths of length two and the way we assemble different ingredients together subject to that.

We organize the rest of the paper as follows. In Section \ref{sec:regularity}, we prove our regularization theorems: Theorem \ref{thm:enhanced-regularization} and Theorem \ref{thm:biregular}. In Section \ref{sec:kst-regular}, we prove Theorem \ref{thm:main-kst}.
In Section \ref{sec:kp-regular}, we prove Theorem \ref{thm:main-kp}. In Section \ref{sec:lower}, we establish corresponding lower bounds.
In Section \ref{sec:conclusion}, we give some concluding remarks.


\section{Proof of Theorem \ref{thm:enhanced-regularization} and Proof of Theorem \ref{thm:biregular}}
\label{sec:regularity}
To prove Theorem \ref{thm:enhanced-regularization} and Theorem \ref{thm:biregular}, we build on Pyber's method from \cite{Pyber} on regular subgraphs as well as a variant of it used by Pyber, R\"odl, and Szemer\'edi \cite{PRS}.
   
\begin{definition}
    {\rm Given a positive integer $d$, a bipartite graph $G$ with a bipartition $(A,B)$ is called $d$-{\it half-regular} at $A$ if all vertices $A$ have degree $d$.}
\end{definition}
The following lemma is a key ingredient in Pyber's approach.

\begin{lemma} [\cite{Pyber}] \label{lem:pyber}
Let $d$ be a positive integer. Let $G$ be a bipartite graph with a bipartition $(A,B)$ such that $|A|\geq|B|$ and $G$ is $d$-half-regular at $A$. Then $G$ contains a matching $M_1$ such that 
$G[V(M_1)]$ is $d$-half-regular at $V(M_1)\cap A$.
\end{lemma}

We are now ready to prove Theorem \ref{thm:enhanced-regularization}.

   \begin{proof}[Proof of Theorem \ref{thm:enhanced-regularization}]
   By a well-known fact, $G$ contains a subgraph $G'$ with minimum degree at least 
   $d:=\lceil cn^{\ve}\rceil$. Let $(A,B)$ be a bipartition of $G'$ with $|A|\geq |B|$.
   Let $G_0$ be obtained from $G'$ by keeping exactly $d$ edges incident to each vertex in $A$.
   Then $G_0$ is $d$-half-regular at $A$. By Lemma~\ref{lem:pyber}, $G_0$ contains a matching $M_1$
   such that $G_0[V(M_1)]$ is $d$-half-regular at $A_1:=V(M_1)\cap A$. Let $B_1=V(M_1)\cap B$.
   Then $|A_1|=|B_1|$.   Let $G_1=G_0\setminus M_1$. Then $(A_1, B_1)$ is a bipartition of $G_1$ and $G_1$ is  $(d-1)$-half-regular at $A_1$ and $|A_1|=|B_1|$. Hence $G_1$ satisfies the hypotheses of Lemma \ref{lem:pyber} with $d$ replaced with $d-1$. We can thus repeat the arguments to find $\ceil{d/2}$ edge-disjoint matchings $M_1,M_2,\dots, M_{\ceil{\frac{d}{2}}}$ in $G$ such that $V(M_1)\supseteq V(M_2)\cdots\supseteq V(M_{\ceil{\frac{d}{2}}})$. Let $\I=\{1,\dots, \ceil{\frac{d}{2}}\}$.  By definition,  for each $i\in \I$, $G_{i-1}[V(M_i)]$ is $(d-i+1)$-half-regular at $A_i$. Hence, $|M_i|\geq d-i+1$.
   In particular, $|M_{\ceil{\frac{d}{2}}}|\geq \frac{d}{2}$. Also, trivially $|M_1|\leq \frac{n}{2}$.
   For every $1\leq j\leq \floor{\log_2 (\frac{2n}{d})}$, let 
   \[\I_j=\{i\in \I:\frac{n}{2^{j+1}} < |M_i| \leq \frac{n}{2^j}.\}\]
   Note that the $\I_j$'s partition $\I$. For each $\I_j$, we call it
   {\it thick} 
    if $|\I_j| \geq (\frac{2^\ve-1}{2})c \left(\frac{n}{2^j}\right)^{\ve}$; otherwise we call it {\it thin}.
    Note that 
    \[\left|\bigcup \{\I_j: \I_j \mbox{ is thin} \}\right|\leq \sum_{j=1}^\infty \left(\frac{2^\ve-1}{2}\right)c \left(\frac{n}{2^j}\right)^\ve    \leq \frac{c}{4}n^\ve\leq d/4.\]
    Hence, 
    \[\left|\bigcup \{\I_j: \I_j \mbox { is thick } \}\right|> d/4.\]

By the pigeonhole principle, there exists some $1\leq j\leq \floor{\log_2 (\frac{2n}{d})}$ such that
$\I_j$ is thick and $|\I_j|\geq d/[4\log_2(\frac{2n}{d})]$. Fix such a $j$. Let $p=|\I_j|$. 
Since $V(M_1)\supseteq V(M_2)\supseteq \dots \supseteq V(M_{\ceil{\frac{d}{2}}})$, 
$\I_j=\{\ell,\ell+1,\dots, \ell+p-1\}$ for some $\ell$. Also, by the definition of $\I_j$,
$|M_i|\geq |M_\ell|/2$ for each $i\in \I_j$.

Let $F$ be the subgraph of $G$ formed by taking
the union of the $M_i$'s that lie in $\I_j$. Then 
$v(F)=v(M_\ell)=2|M_\ell|$ and $e(F)\geq p(|M_\ell|/2)=pv(F)/4$. 
Hence, $F$ has average degree $d(H)\geq p/2$. Also, since $F$
is the union of $p$ matchings, $\Delta(F)\leq p$. By iteratively deleting a vertex whose degree drops below $p/6$, we can find a subgraph $H$ of $F$ with $e(H)\geq pv(F)/12=p|M_\ell|/6$ and $\delta(H)\geq p/6$. Let $m=v(H)$.
Now, $H$ is an $m$-vertex $6$-almost-regular graph. Also, $m=v(H)\leq 2|M_\ell|$.
Since $\I_j$ is thick, $p=|\I_j|\geq  (\frac{2^\ve-1}{2})c \left(\frac{n}{2^j}\right)^{\ve}
\geq \frac{2^\ve-1}{2}c|M_\ell|^\ve$. So,
\[e(H)\geq \frac{p}{6}|M_\ell|\geq \left(\frac{2^\ve-1}{12}\right)c |M_\ell|^{1+\ve}
\geq \frac{(2^\ve-1)c}{48}  m^{1+\ve}.\]

Also, by our choice of $\I_j$, we have $p=|\I_j|\geq d/[4\log_2(\frac{2n}{d})]$.
So $d(H)\geq p/3\geq d/(12 \log_2(\frac{2n}{d})]$.
\end{proof}

Next, we prove Theorem \ref{thm:biregular}. As in the proof of Theorem \ref{thm:enhanced-regularization}, our first step is to reduce the host graph to a half-regular graph.
This is accomplished by the following simple lemma, which might be of independent interest.

\begin{lemma}\label{lem:onesideregular}
Let $c, \alpha,  \beta>0$ be reals, where $\alpha, \beta \leq 1$ and $\alpha + \beta \geq 1$.
    Let $G$ be a bipartite graph with parts $M,N$, where $|M|\leq |N|$.
    Suppose $e(G) \geq c|M|^\alpha|N|^{\beta}$ and $d(G)\geq 8$. Then there exists $G' \subseteq G$ with parts $M' \subseteq M$, $N' \subseteq N$ such that  
    $e(G')\geq \frac{1}{2^{2+\frac{1}{\alpha}+\frac{1}{\beta}}} c|M'|^\alpha|N'|^{\beta}$, $G'$ is  half regular at the larger side of $(M',N')$, and $d(G')\geq \frac{d(G)}{8}$.
\end{lemma}
\begin{proof} Let $d=d(G)$.
    Let $G_0$ be a minimal subgraph of $G$ satisfying  $e(G_0) \geq c |V(G_0) \cap M|^\alpha |V(G_0)\cap N|^\beta$ and $d(G_0)\geq d$. 
    Since $G$ satisfies the condition, $G_0$ exists. Let $M_0=V(G_0)\cap M, N_0=V(G_0)\cap N$.
    If $\delta(G_0)\geq \frac{1}{4}d(G_0)$, then we let $G'$ be obtained from $G_0$ by keeping exactly $\ceil{\frac{1}{4}d(G_0)}$ edges at each vertex in the larger side of $(M_0,N_0)$.
    Then $e(G')\geq \frac{1}{4}d(G_0)\cdot \frac{1}{2}(|M_0|+|N_0|)\geq \frac{1}{4}e(G_0)
    \geq \frac{c}{4}|M_0|^\alpha |N_0|^\beta$ and $d(G_1)\geq \frac{1}{4} d(G)$. The claim clearly holds. So, we may assume $\delta(G_0)<\frac{1}{4} d(G_0)$.
    
    Let $G_1$ be obtained from $G_0$ by iteratively removing vertices whose degrees become less than $\frac{1}{4}d(G_0)$. Then $e(G_1)\geq \frac{1}{2}e(G_0)$ and $\delta(G_1)\geq \frac{1}{4}d(G_0)$. Let $M_1=V(G_1)\cap M, N_1=V(G_1)\cap N$. Then $e(G_1)\geq \max\{\frac{c}{2}
    |M_0|^\alpha|N_0|^\beta$, $\frac{1}{4}d(G_0)(|M_0|+|N_0|)\}$. On the other hand, by the minimality of $G_0$, $e(G_1)<\max\{ c|M_1|^\alpha |N_1|^\beta, \frac{d}{2}(|M_1|+|N_1|)\}$.
    This implies either $c|M_1|^\alpha|N_1|^\beta> \frac{c}{2} |M_0|^\alpha|N_0|^\beta$ or
    $\frac{d}{2}(|M_1|+|N_1|)\geq \frac{d(G_0)}{4}(|M_0|+|N_0|)$. 
    Let $G'$ be obtained from $G_1$ by keeping exactly $\frac{1}{4}d(G_0)$ edges at each vertex in the larger side of $(M_1,N_1)$. Let $M'=M_1, N'=N_1$.
    In the former case, we have
    $|M_1|\geq \frac{1}{2^{\frac{1}{\alpha}}}|M_0|$ and  $|N_1|\geq \frac{1}{2^{\frac{1}{\beta}}}|N_0|$. So, $\max\{|M_1|,|N_1|\}\geq \frac{1}{2}\frac{1}{2^{\frac{1}{\alpha}+\frac{1}{\beta}}}(|M_0|+|N_0|)$.  In the latter case, we have $\max\{M_1,N_1\}\geq \frac{1}{4}(|M_0|+|N_0|)$.
    Hence, $e(G')\geq \frac{1}{4}d(G_0)\cdot \max\{|M_1|,|N_1|\}\geq\frac{1}{8}d(G_0)\frac{1}{2^{\frac{1}{\alpha}}+2^{\frac{1}{\beta}}} (M_0|+|N_0|)$ $\geq \frac{1}{2^2+2^\frac{1}{\alpha}+2^{\frac{1}{\beta}}} e(G_0)\geq \frac{1}{2^2+2^\frac{1}{\alpha}+2^{\frac{1}{\beta}}} c|M'|^\alpha |N'|^\beta$. Also, $d(G')\geq \frac{1}{8}d(G_0)\geq \frac{1}{8}d(G)$.

\end{proof}

As defined in \cite{PRS}, given a bipartite graph $G$ with an ordered bipartition $(M,N)$, an {\it $N$-roof}  $R$ is a subgraph $R \subseteq G$ such that every vertex $v \in N$ satisfies that $d_R(v) = 1$. 
We need the following lemma from \cite{PRS}.

\begin{lemma} [\cite{PRS}] \label{lem:roof}
    Let $G$ be a bipartite graph with an ordered bipartition $(M, N)$. 
    Then \[\min \{\Delta(R): R \mbox{ is an $N$-roof of $G$}\} = 
    \max_{X\subseteq N} \left\lceil\frac{|X|}{|N_G(X)|}\right \rceil.\]
\end{lemma}

\begin{lemma}\label{lem:halfreg-almreg}
Let $0<\alpha, \beta \leq 1$ be reals such that $\alpha+\beta >  1$. Let $c > 0$. 
    Let $G$ be a bipartite graph with an ordered bipartition $(M, N)$ such that $|M|\leq |N|$,
    $G$ is $d$-half-regular at $N$, where $d\geq 1$, and
    $e(G)\geq c|M|^\alpha |N|^\beta$. Then $G$ contains a $16$-almost-biregular subgraph $G'$ with parts $M'\subseteq M, N'\subseteq N$ such that 
    \[e(G')\geq \frac{(2^{\alpha+\beta-1}  -1 )c}{16} |M'|^\alpha |N'|^\beta \mbox{ and } d(G')\geq \frac{1}{8\log |M|} d(G).\]
\end{lemma}
   \begin{proof}
   Let $\mu=|N|/|M|$ and $t=\ceil{\mu}$. Since $\mu\geq 1$, we have $t\leq 2\mu$. Let $N_1$ be a minimal subset of $N$ satisfying $|N_1|/|N_G(N_1)|\geq\mu$. Since $N$ satisfies this condition, $N_1$ exists. Let $M_1=N_G(N_1)$.
   Let $G_1=G[N_1\cup M_1]$.  By our choice of $N_1$, for every subset $X$ of $N_1$ we have
$|X|/N_{G_1}(X)|=|X|/N_G(X)|\leq \mu$. By Lemma \ref{lem:roof}, $G_1$ contains an $N_1$-roof
$R_1$ with $\Delta(R_1)\leq \ceil{\mu}=t$. Let $H_1=G_1\setminus E(R_1)$. Then, $H_1$ is a 
bipartite graph with parts $N_1$ and $M_1$ that is $(d-1)$-half-regular at $N_1$.
Recall that $|N_1|/|M_1|\geq \mu$ by our choice of $N_1$. Hence, in particular $|N_1|/ |N_{H_1}(N_1) |\geq |N_1| / |M_1| \geq \mu$. 
Now, let $N_2$ be a minimal subset of $N_1$ such that $|N_2|/|N_{H_1}(N_2)|\geq \mu$.
Since $N_1$ satisfies the condition, $N_2$ exists. Let $M_2=N_{H_1}(N_2)$. 
Let $G_2=H_1[N_2\cup M_2]$. By our choice of $N_2$, for every subset $X\subseteq N_2$, we have
$|X|/|N_{G_2}(X)|=|X|/|N_{H_1}(X)|\leq \mu$. By Lemma \ref{lem:roof}, $G_2$ contains an $N_2$-roof $R_2$ with $\Delta(R_2)\leq \ceil{\mu}=t$. Let $H_2=G_2\setminus E(R_2)$. Note that
$|N_2|/|N_{H_2}(N_2)|\geq |N_2|/|M_2|\geq \mu$.
We can continue the process to find a sequence of subgraphs $G_1\supseteq G_2\dots \supseteq G_d$
such that for each $i\in [d]$, $G_i$ is $( d - i + 1)$-half-regular at $N_i=V(G_i)\cap N$, $G_i$ contains
an $N_i$-roof $R_i$ with $\Delta(R_i)\leq t$ and for each $i\in [d-1]$, $G_{i+1}\subseteq G_i\setminus E(R_i)$.
Furthermore, for each $i\in [d]$, $|N_i|/|M_i|\geq \mu=|N|/|M|$.

Note that $|N_1|\leq |N|$ and $|N_d|\geq |N|/|M|$. For each $1\leq j\leq \ceil{\log |M|}$, let
$\I_j=\{i: \frac{|N|}{2^j}\leq |N_i|< \frac{|N|}{2^{j-1}}\}$. Then $\bigcup_{j=1}^{\ceil{\log |M|}} \I_j=[d]$. Let 
\begin{equation}\label{eq:ve-definition}
\ve=\frac{1}{2}(2^{\alpha+\beta-1}-1).
\end{equation}
For each $1\leq j \leq \ceil{\log |M|}$, we call $\I_j$ {\it thin} if $|\I_j|< \frac{\ve d}{2^{(\alpha+\beta-1) j}}$ and {\it thick} otherwise.

    Note that 
    \[\left|\bigcup \{\I_j: \I_j \mbox{ is thin} \}\right|\leq \sum_{j=1}^\infty 
    \frac{\ve d}{2^{(\alpha+\beta-1) j}}\leq \frac{d}{2}.\]
    Hence,
    \[\left|\bigcup \{\I_j: \I_j \mbox { is thick } \}\right|\geq \frac{d}{2}.\]

By the pigeonhole principle, there exists some $1\leq j\leq \ceil{\log |M|}$ such that
$\I_j$ is thick and $|\I_j|\geq \frac{d}{2\ceil{\log |M|}}$. Fix such a $j$ and let $p=|\I_j|$. Then
\begin{equation}\label{eq:p-lower}
p\geq \max\{\frac{\ve d}{2^{(\alpha+\beta-1) j}}, \frac{d}{2\ceil{\log |M|}}\}.
\end{equation}
Since $N_1\supseteq N_2\supseteq \dots \supseteq N_d$, 
$\I_j=\{\ell, \ell+1,\dots, \ell+p-1\}$ for some $\ell$.
Let $R=\bigcup_{i=\ell}^{\ell+p-1} R_i$. Then $R$ is a bipartite subgraph of $G$ with parts
$N_\ell$ and $M_\ell$. For each $i\in [\ell, \ell+p-1]$, 
$R_i$ is $1$-half-regular at $N_i$ and $|N_i|\geq |N_\ell|/2$. Hence,
\begin{equation}\label{eq:ER-lower}
e(R)=\sum_{i=\ell}^{\ell+p-1}|N_i|\geq (p/2)|N_\ell|.
\end{equation}
Let $d_M=E(R)/|M_\ell|$, $d_N=E(R)/|N_\ell|$, $\Delta_M=\max\{d_R(x): x\in M_\ell\}$
and $\Delta_N=\max\{d_R(y): y\in N_\ell\}$.
Then $d_M\geq (p/2)|N_\ell|/|M_\ell|\geq (p/2)\mu$ and $d_N\geq (p/2)$.
Also, $\Delta_M\leq pt\leq 4d_M$ and $\Delta_N\leq p\leq 2d_N$. 
Let $G'$ be obtained from $R$ by iteratively deleting vertices in $M_\ell$ whose degrees
become less than $d_M/4$ and vertices in $N_\ell$ whose degree become less than $d_N/4$. Then $e(G')\geq (1/2)e(R)$ and $G'$ is $16$-almost-biregular. 
We next argue that $e(G')\geq \frac{1}{16}(2^{\alpha+\beta-1}-1) c|M_\ell|^\alpha |N_\ell|^\beta$.  Observe that 
\[ d\geq c|M|^\alpha |N|^{\beta-1}= c(|M|/|N|)^\alpha |N|^{\alpha+\beta-1}=\frac{c}{\mu^\alpha} |N|^{\alpha+\beta-1}.\]
By  \eqref{eq:p-lower} and \eqref{eq:ER-lower}, and that $\frac{|N|}{2^j} \leq |N_\ell|< \frac{|N|}{2^{j-1}}$, we have
\begin{equation}\label{eq:G'-bound}
e(G')\geq \frac{1}{2}e(R)\geq \frac{1}{4} p|N_\ell|\geq \frac{1}{4} \frac{\ve d}{2^{(\alpha+\beta-1) j}}|N_\ell|\geq  
\frac{\ve}{4}\frac{c}{\mu^\alpha} \frac{|N|^{\alpha+\beta}}{2^{(\alpha+\beta)j}}
\geq \frac{\ve}{16} \frac{c}{\mu^\alpha} |N_\ell|^{\alpha+\beta}.
\end{equation}
Recall that $\frac{|N_\ell|}{|M_\ell|}\geq \mu$. Hence $\frac{1}{\mu}\geq \frac{|M_\ell|}{|N_\ell|}$.
By \eqref{eq:ve-definition} and \eqref{eq:G'-bound}, we have
\[e(G')\geq \frac{\ve c}{16} |M_\ell|^\alpha |N_\ell|^\beta 
=\frac{( 2^{\alpha+\beta-1}  - 1)c}{32} |M_\ell|^\alpha |N_\ell|^\beta.\]
Also, since $e(G')\geq (p/4)|N_\ell$ and $v(G')\leq 2|N_\ell|$, $d(G')\geq p/4\geq \frac{d}{8\ceil {\log |M|}}\geq \frac{d(G)}{8\ceil{\log |M|}}$. 
\end{proof}
Now we are ready to prove Theorem \ref{thm:biregular}.

\begin{theorem}[Theorem \ref{thm:biregular} restated]
    Let $0<\alpha, \beta\leq 1$ be reals satisfying  $\alpha+\beta>1$.
    There exists a positive constant $\lambda=\lambda(\alpha,\beta)$
    such that the following holds. 
    Let $G$ be a bipartite graph with an ordered partition $(M, N)$, with $|M| = m$, $|N| = n$, and $m \leq n$ such that $e(G) \geq cm^{\alpha}n^{\beta}$ and $d(G)\geq 8$. Then $G$ contains a $16$-almost biregular subgraph $G'$, with a partition $(M',N')$ where $M' \subseteq M, N' \subseteq N$, $|M'| = m', |N'| = n'$ such that  $e(G') \geq  \lambda c (m')^{\alpha}(n')^{\beta}$ and $d(G') \geq \frac{d(G)}{64\log m}$. 
\end{theorem}
\begin{proof}
Let $\lambda= \frac{(2^{\alpha + \beta - 1} - 1)  }{2^{6 + \frac{1}{\alpha}  + \frac{1}{\beta}}}$.
 Applying Lemma \ref{lem:onesideregular}, $G$ contains a subgraph
$G_0$ with parts $M_0\subseteq M, N_0\subseteq N$ such that $G$ is half-regular at the larger side of $(M_0,N_0)$ and $e(G_0)\geq \frac{c}{2^{2 + \frac{1}{\alpha} + \frac{1}{\beta}}}m_0^\alpha n_0^\beta$, and $d(G_0) \geq \frac{1}{8} d(G)$ where $m_0=|M_0|, n_0=|N_0|$.
If $n_0\geq m_0$, then we apply Lemma~\ref{lem:halfreg-almreg} to $G_0$ and the ordered partition $(M_0,N_0)$. If $m_0\geq n_0$, then we apply Lemma \ref{lem:halfreg-almreg} to $G_0$ and the ordered partition $(N_0,M_0)$. In either case, we obtain
a $16$-almost- biregular graph $G'$  with  parts $M'\subseteq M_0$ and $N'\subseteq N_0$ such that $e(G') \geq \frac{(2^{\alpha + \beta - 1} - 1) c }{2^{6 + \frac{1}{\alpha}  + \frac{1}{\beta}}} (m')^{\alpha}(n')^{\beta}$ and that $d(G')\geq \frac{d(G_0)}{8\log m_0} \geq \frac{d(G)}{64 \log m}$.
\end{proof}

For applications in this paper, we will use the following corollary of Theorem \ref{thm:biregular}.

\begin{corollary} \label{cor:biregular}
    Let $0<\alpha, \beta\leq 1$ be reals satisfying  $\alpha+\beta>1$.
    There exists a  positive constant $\lambda'=\lambda'(\alpha,\beta)$ 
    such that the following holds. 
    Let $G$ be a bipartite graph with an ordered partition $(M, N)$, with $|M| = m$, $|N| = n$, and $m \leq n$ such that $e(G) \geq c(m^{\alpha}n^{\beta}+n\log m)$. Then $G$ contains a $16$-almost biregular subgraph $G'$, with a partition $(M',N')$ where $M' \subseteq M, N' \subseteq N$, $|M'| = m', |N'| = n'$ such that  $e(G') \geq  \lambda' c[(m')^{\alpha}(n')^{\beta}+ m'+n']$.    
\end{corollary}
\begin{proof}
By Theorem \ref{thm:biregular}, there exists a subgraph $G'$ with
$e(G') \geq  \lambda c (m')^{\alpha}(n')^{\beta}$ and $d(G') \geq \frac{d(G)}{64 \log m}\geq \frac{c\log m}{64 \log m}=\frac{c}{64}$. Hence, $e(G')\geq \frac{1}{2}[\lambda c (m')^{\alpha}(n')^{\beta}+\frac{c}{128}(m'+n')]\geq \lambda'c[(m')^\alpha (n')^\beta+m'+n']$, where $\lambda'=\min\{\frac{\lambda}{2}, \frac{1}{256}\}$.    
\end{proof}

\section{Proof of Theorem \ref{thm:main-kst}}
\label{sec:kst-regular}
To prove Theorem \ref{thm:main-kst}, we first consider the problem in the setting when the host graph is almost-bi-regular with large minimum degree. For a bipartite graph $G$ with a partition $(M,N)$, we let $d_M(G)$ denote the minimum degree of a vertex in $M$ and $d_N(G)$ the minimum degree of a vertex in $N$. When the context is clear, we will just write $d_M,d_N$ for $d_M(G), d_N(G)$, respectively.

We now introduce some definitions that will be used throughout the section. Some definitions originated in \cite{CJL} and followup works, some are new.
Let $T$ be a tree and and $v$ a nonleaf of $T$. Suppose $v$ has degree $m$.
Then there is a unique collection of edge-disjoint subtrees $T_1,\dots T_m$
each of which has $v$ as a leaf and whose union is $T$. We call these $m$ subtrees the subtrees of $T$ {\it split by $v$}.

\begin{definition} [admissible trees, linked tuples]
Let $G$ be a graph. Let $h\geq 2$ be a fixed positive integer. Let $T$ be a tree whose vertices are labeled with $\{1,\dots, v(T)\}$. We will view each subtree of $T$ also as a vertex-labeled tree.
If $v_1,\dots, v_p$ are the leaves of $T$ in the natural order determined by their labels, we call $\langle v_1,\dots, v_p \rangle$
the {\it leaf vector} of $T$. If $F$ is a copy of $T$ in $G$, we call image of the leaf vector of $T$ in $F$ the {\it leaf vector} of $F$. Given a family $\F$
of copies of $T$ in $G$ and a member $F\in \F'$, we say that  $F$ is {\it $h$-heavy} in $\F$ if there are at least $h^{e(T)^2}$ members of $\F$ that share the same leaf vector with $F$ and we say that $F$ is $h$-light in $\F$ if it is not $h$-heavy in $\F$.

Given a subtree $D$ of $T$ and  $F\in \F$, we let $F[D]$ denote the image of $D$ in $F$ and call it the {\it $D$-projection of $F$}. We also let $\F[D]=\{F[D]: F\in \F\}$. 
We say that $F$ is {\it $h$-admissible} in $\F$ if either $F=K_2$ or otherwise for every nonleaf $v$ of $F$ each 
subtree of $F$ split by $v$ is $h$-light in $\F$. Given 
a  $p$-tuple $\langle y_1,\dots, y_p\rangle$ of vertices in $G$, we say that $\langle y_1,\dots, y_p\rangle $ is {\it $(T,h)$-linked} in $\F$ if there are $h$ members of $\F$ that have the leaf vector $\langle y_1,\dots, y_p\rangle$ but are otherwise vertex-disjoint. In the special case of $T$ being a $k$-path $P_{k+1}$, instead of saying a tuple $\langle x,y\rangle $ is $(P_{k+1},h)$-linked in $\F$, we will slightly abuse notation and just say that $x,y$ are {\it $(k,h)$-linked} in $\F$. If $\F$ is the family of all copies of $T$ in $G$, then instead of saying a member $F$ or a leaf vector has a certain property in $\F$ we will just say the member or the leaf vector has that property in $G$.
\end{definition}

We next give a few useful lemmas on admissible trees and linked tuples.

\begin{lemma} \label{lem:admissible-paths}
  Let $\ell,\eta\geq 2$ be integers. Let $P$ be vertex-labeled path of length $\ell$.  Let $G$ be a graph. 
  If $F$ is a copy of $P$ that is $\eta$-heavy in $G$, then for some $2\leq j\leq \ell$, $F$ contains a $j$-path that is $\eta$-heavy and $\eta$-admissible in $G$.  
\end{lemma}
\begin{proof}
    Let $F$ be a copy of $P$ in $G$ that is $\eta$-heavy in $G$. Let $P_0$ be a minimal subpath of $P$ such that $F[P_0]$ is $\eta$-heavy in $G$. Suppose $F[P_0]$ is not $\eta$-admissible, then there exists some nonleaf $w$ in it such that one of the two subpaths split by $w$ is $\eta$-heavy in $G$, contradicting our choice of $P_0$. So $F[P_0]$ is $\eta$-admissible in $G$. 
\end{proof}

Recall that a {\it spider} $T$ with $p$ legs is a tree consisting of $p$ paths sharing one common endpoint $u$ but are otherwise vertex-disjoint. If $\langle v_1,\dots, v_p\rangle$ is the leaf vector of $T$ and $\ell_i$ is the length of the $u,v_i$-path for each $i\in [p]$, we call $\langle \ell_1,\dots, \ell_p\rangle$ the {\it length vector} of $T$.

\begin{lemma} \label{lem:admissible-spiders}
Let $k,s\geq 2$ be integers. Let $T$ be a vertex-labled spider of $s$ legs each having length $k$.
Let $G$ be a graph. Let $\F$ be a family of copies of $T$ in $G$ each of which contains no $j$-path that is $\eta$-heavy in $G$,
  for any $2\leq j\leq k$.
  If $F$ is a member of $\F$ that is $\eta$-heavy in $\F$
  then there exists an $s$-legged subspider $D$ of $T$ such that $F[D]$ 
  is $\eta$-heavy and $\eta$-admissible in $\F[D]$. Furthermore, the sum of the lengths of any two legs of $D$ is larger than $k$, and thus $D$ has at most one leg of length one.
  \end{lemma}
  \begin{proof}
    Let $F$ be a member of $\F$ that is $\eta$-heavy in $\F$.
    If $F$ is $\eta$-admissible in $\F$ then we are done by taking $D=T$.
    So, suppose $F$ is not $\eta$-admissible in $\F$.
    By definition, there exists a nonleaf $v$ of $T$ such that  for some  subtree $D$ of $T$ split by $v$,
    $F[D]$ is is $\eta$-heavy in $\F[D]$. Since $F$ contains no $j$-path that is $\eta$-heavy in $G$, for any $2\leq j\leq k$,  $v$ must be an internal vertex on some leg of $T$ and $D$ is an $s$-legged subspider of $T$. Let $D$ be a minimal $s$-legged subspider of $T$ such that $F[D]$ is $\eta$-heavy in $\F[D]$.
    By the minimality of $D$, $F[D]$ is $\eta$-heavy and $\eta$-admissible in $\F[D]$. 
Let $L_1,\dots, L_s$ be the legs of $D$ with lengths $\ell_1,\dots, \ell_s$, respectively.  Let $\langle y_1,\dots, y_s\rangle$ be the leaf vector of $F[D]$. Consider any $i,j\in [s], i\neq j$.
Suppose the number of members of $\F[L_i\cup L_j]$ that have leaf vector $\langle y_i,y_j\rangle$ is less than $\eta^{(\ell_i+\ell_j)^2}$. Then since members of $\F$ contain no $j$-path that is $\eta$-heavy in $G$, for $2\leq j\leq k$, the number of members of $\F[D]$ 
that have leaf vector $\langle y_1,\dots, y_s\rangle$ is less than $\eta^{(\ell_i+\ell_j)^2}\cdot \prod_{q\in [s]\setminus\{i,j\}} \eta^{\ell_q^2}<\eta^{e(D)^2}$, contradicting $F[D]$ being $\eta$-heavy in $\F[D]$. Hence, in particular, $F[L_i\cup L_j]$ is a $\eta$-heavy $(\ell_i+\ell_j)$-path in $G$. But $F$ does not contain any $\eta$-heavy  $j$-path in $G$, for any $2\leq j\leq k$.
So, we must have $\ell_i+\ell_j>k$. In particular, there can be at most one $i$ with $\ell_i=1$.
\end{proof}

The following lemma is used in most of the recent works on subdivisions (see \cite{CJL, Janzer-bipartite, Janzer-longer, Jiang-Qiu, Jiang-Qiu2} etc) for trees $T$ that are paths or spiders. We give a more general version here.

\begin{lemma} \label{lem:Th-linked}
    Let $G$ be a graph. Let $h,t,p$ be integers, where $t\geq p \geq 2$ and $h\geq t$. Let $T$ be a vertex-labeled tree with $t$ edges  and $p$ leaves.
    Let $\F$ be a collection of copies of $T$ in $G$ such that all the members are $h$-admissible in $\F$ and share the same leaf vector $\langle y_1,\dots, y_p\}$.
    If $|\F|\geq (1/2) h^{t^2}$, then $\langle y_1,\dots, y_p\rangle$ is $(T,h)$-linked in $\F$.
\end{lemma}
\begin{proof} First note that the claim holds trivially when $t=2$. So, we will assume $t\geq 3$.
    Let $\{F_1,\dots, F_m\}$ be a maximum subcollection of  members of $\F$  that are pairwise vertex disjoint outside $\{y_1,\dots, y_p\}$. If $m\geq h$, then we are done. Hence we may assume that $m<h$. Let $S=\bigcup_{i=1}^m V(F_i)\setminus \{y_1,\dots, y_p\}$. 
     By our assumption $|S|<ht$.
    By the maximality of $\{F_1,\dots, F_m\}$, each member of $\F$ maps some nonleaf vertex of $T$ to a vertex in $S$. Consider any nonleaf $w$ of $T$ and any vertex $v\in S$. Suppose $w$ has degree $\ell$ in $T$, and that $T_1,\dots, T_\ell$ are the subtrees of $T$ split by $w$.
    Consider any member $F$ of $\F$  that maps $w$ to $v$. Since $F$ is $h$-admissible in $\F$, by definition,  $F[T_1],\dots, F[T_\ell]$ are $h$-light in $\F$.
     Hence, the number of members of $\F$ that map $w$ to $v$ is at most $\prod_{i=1}^\ell h^{e(F_i)^2}<h^{(t-1)^2+1}$.  
    Hence, $|\F|< |S| t h^{t^2-2t+2}<t^2h\cdot h^{t^2-2t+2}<(1/2) h^{t^2}$, where we use the condition $h\geq t$. But this contradicts our assumption about $\F$, completing the proof.
\end{proof}

The next notion is implicit in earlier works, which we will make explicit here.

\begin{definition} [$\eta$-robust family]
Let $G$ be a graph.
 Let $T$ be a vertex-labeled tree. Let $\eta$ be a positive integer. Let $\F$ be a family of copies of $T$ in $G$.
 We say that $\F$ is {\it $\eta$-robust} if 
the leaf vector of each member of $\F$ is $(T,\eta)$-linked in $\F$ and 
for each pair of subtrees $T_1,T_2$ of $T$ satisfying $T_1\subseteq T_2, |T_2|=|T_1|+1$, among the members of $\F$ having the same $T_1$-projection as $F$ there are at least $\eta$ different $T_2$-projections.
\end{definition}

\begin{lemma}[Robust subfamily] \label{lem:robust-subfamily}
    Let $\ve, \mu>0$ be reals with $\mu\geq 1$. Let $\eta \geq 2$ be an integer.
    Let $T$ be a vertex-labeled tree with an ordered partition $(A,B)$, where
    $|A|=a, |B|=b$. There exists a constant $C_1 = C_1(\ve,\mu, \eta, T) >0$ such that the following holds. 
    Let $G$ be an $\mu$-almost-biregular graph with an ordered partition $(M,N)$ and minimum degree at least $C_1$.
    Let $\F$ be a family of copies of $T$ in $G$ with $A$ embedded into $M$ where each member is $\eta$-heavy and $\eta$-admissible in $\F$. If $|\F|\geq \ve e(G)d_M^{b-1}d_N^{a-1}$,
    then $\F$ contains an $\eta$-robust subfamily $\F'$ satisfying 
    $|\F\setminus \F'|\leq \ve  e(G)d_M^{b-1}d_N^{a-1}$. In particular, if $|\F| \geq 2\ve e(G)d_M^{b-1}d_N^{a-1}$, then $|\F'| \geq \frac{1}{2} |\F|$.  
\end{lemma}
\begin{proof} 
Let $t=e(T)$ and let
\begin{equation} \label{eq:C1-choice}
C_1=\frac{( a + b)(2\mu)^{a+b}\eta^{t^2+1}}{\ve^2}.
\end{equation}
Suppose $G$ and $\F$ satisfy the given conditions. By our assumption,  for each member $F$ of $\F$, there are at least $\eta^{t^2}$ members of $\F$ sharing the same leaf vector as $F$.
    Starting with $\F$ we iterate the following steps. Whenever there exists a member $F\in \F$ 
    and a pair of subtrees $T_1,T_2$ of $T$ with $T_1\subseteq T_2, |T_2|=|T_1|+1$ such that over all the remaining members $F'$ of $\F$ satisfying $F'[T_1]=F[T_1]$, there are fewer than $\eta$ different images of $F'[T_2]$, we say that these members become {\it $(T_1,T_2)$-bad} and remove
    them from $\F$, including $F$ itself. We will call this type of removals  type one removals.
Whenever there exists a member $F$ of $\F$ such that the number of remaining members of $\F$ that share the same leaf vector with $F$ is less than $(1/2)\eta^{t^2}$, we remove all these members (including $F$). We call this type of removals type two removals.     

    {\bf Claim 1.} For each pair of subtrees $T_1, T_2$ of $T$ with $T_1\subseteq T_2, |T_2|=|T_1|+1$, the number of members of $\F$ that become $(T_1, T_2)$-bad in the removal process is at most  
    $(\eta/C_1)e(G)\mu^{a+b-3} d_M^{b-1} d_N^{a-1}$.

    \begin{proof} [Proof of Claim 1.]
    Let $a_1=|A\cap V(T_1)|, b_1=|B\cap V(T_1)|$. We may designate any edge $uv$ of $T_1$ as a first edge, where $u\in A, v\in B$.
    There are $e(G)$ ways to embed $uv$ in $G$ such that $u\in M, v\in N$. Then there are most $(\mu d_M)^{b_1-1} (\mu d_N)^{a_1-1}$ ways to embed the rest of $T_1$ in $G$. So the number of different
    $T_1$-projections of members of $\F$ with $A\cap T_1$ embedded in $M$ is at most $e(G)(\mu d_M)^{b_1-1} (\mu d_N)^{a_1-1}$.
    Let $w$ denote the leaf of $T_2$ not in $T_1$. First suppose that $w\in A$. Let $D$ be a fixed $T_1$-projection of members of $\F$.
    Consider members $F$ of $\F$ with $F[T_1]=D$ that become $(T_1,T_2)$-bad at some point . By definition, these members have fewer than $\eta$ different $T_2$-projections at the point they become $(T_1,T_2)$-bad.
    So, there are at most $\eta (\mu d_M)^{b-b_1} (\mu d_N)^{a-a_1-1}$ such members. 
    So,  the number of members that become $(T_1, T_2)$-bad is at most 
    \begin{eqnarray*}
        &&e(G)(\mu d_M)^{b_1-1} (\mu d_N)^{a_1-1}\cdot \eta\cdot (\mu d_M)^{b-b_1}(\mu d_N)^{a-a_1 - 1}\\
        &
    =&(\eta/d_N) e(G)\mu^{a+b-3} d_M^{b-1} d_N^{a-1} \leq (\eta/C_1) e(G)\mu^{a+b-3} d_M^{b-1} d_N^{a-1},
    \end{eqnarray*}
    where we used the fact that $d_N\geq C_1$. A similar argument can be applied if $w\in B$.
    \end{proof}
    Since there are clearly at most $(a + b)2^{a+b}$ pairs of subtrees $T_1,T_2$ satisfying $T_1\subseteq T_2$
    and $|T_2|=|T_1|+1$,  the total number of members of $\F$ that are removed by type one removals is at most 
    \begin{equation}\label{eq:type-one-bound}
    s_1:=(a + b)2^{a+b}(\eta/C_1)e(G)\mu^{a+b-3} d_M^{b-1} d_N^{a-1}.
    \end{equation}
Next, observe that  when a member $F\in \F$ gets removed in a type two removal, some member of $\F$
sharing the same leaf vector must have been removed earlier in a type one removal, or else the number of members having that leaf vector would not have dropped below $(1/2)\eta^{t^2}$.
So, the number of members removed in type two removals is at most $ (1/2)\eta^{t^2} s_1$.
Hence, the total number of members of $\F$ that are removed is at most 
\[(1+(1/2)\eta^{t^2}) s_1\leq \frac{(a + b)(2\mu)^{a+b}\eta^{t^2+1}}{C_1} e(G) d_M^{b-1} d_N^{a-1}\leq \ve|\F|,
\]
as $C_1=\frac{(a + b)(2\mu)^{a+b}\eta^{t^2+1}}{\ve^2}$ and $|\F|\geq \ve e(G) d_M^{b-1}d_N^{a-1}$.
Let $\F'$ denote the remaining subfamily at the end of the removal process.
Since no more type two removals can be performed, for each $F\in \F'$, there are at least $(1/2)\eta^{t^2}$ members of $\F'$ having the same leaf vector as $F$. By Lemma \ref{lem:Th-linked}, this leaf vector is $(T,\eta)$-linked. Also, since no more type one removal can be performed, $\F'$ satisfies both conditions of being $\eta$-robust.
\end{proof}

The next three lemmas show that robust families of copies of a tree $T$ have rich structures, a property that was first utilized by Janzer \cite{Janzer-bipartite} for the case where $T$ is a spider (see also \cite{Jiang-Qiu}).
Here, we give a more general version that will be applied to both paths and spiders, both of which 
are useful for building subdivisions. 

\begin{lemma} \label{lem:robust-structure}
Let $\eta,t$ be integers, where $\eta\geq 2t$ and $t\geq 0$.
Let $T$ be vertex-labeled tree with leaf vector
$\langle v_1,\dots, v_p\rangle$, where $p\geq 2$. Let $j\in [p]$.
Let $G$ be a graph. Let $\F$ be a $\eta$-robust family of copies of $T$ in $G$.
Let $F$ be any member of $\F$ with leaf vector $\langle y_1,\dots, y_p\rangle$.
Let $x_j$ denote the unique neighbor of $y_j$ in $F$.
Let $S$ be a set of vertices in $G$ not containing $y_j$. 
Suppose $\eta \geq |S|+2t+2$. Then there exists a member $F'\in \F$ with
leaf vector $\langle y_1,\dots, y_{j-1}, y'_j,y_{j+1},\dots, y_p\rangle$,
a $y_j,y'_j$-path $Q$ of length $2t$ in $G$ such that $V(Q)$ is disjoint from $S$.
There also exists a member $F''\in \F$ with leaf vector $\langle y_1,\dots, y_{j-1}, y''_j, y_{j+1},\dots, y_p\rangle$ and a $x_j,y''_j$-path $R$ of length $2t+1$ in $G$ such that
$V(R)$ is disjoint from $S\cup y_j$.
\end{lemma}
\begin{proof} Since $\F$ is $\eta$-robust and $\eta\geq |S|+2t+2$, there exists a member $F^*$ of $\F$
that contains $T-y_j$ but maps $v_j$ to some vertex $y^*_j\notin S\cup\{ y_j\}$.
If $t=0$, then the statements of the lemma hold by taking $F'=F$ and $F''=F^*$.
Hence, we may assume that $t\geq 1$.
Let $u_j$ denote the unique neighbor of $v_j$ in $T$ (i.e. the preimage of $x_j$ in $T$).
    Let $F_0=F$. Since $\F$ is $\eta$-robust, there exist $\eta$ internally disjoint members of $\F$ with leaf vector $\langle y_1,\dots, y_p\rangle$. Since $\eta \geq |S| + 2t +2$, among these members we can find a member $F_1$ that maps $u_j$ to a vertex $x_j^1\notin S\cup\{y_1,\dots, y_p\}\cup\{ x_j\}$.
    Since $\F$ is $\eta$-robust, we can find a member $F_2$ with $F_2[T-\{v_j\}]=F_1[T-\{v_j\}]$ that maps $v_j$ to a vertex $y_j^1$ outside $S\cup \{y_1,\dots, y_p\}\cup \{x_j,x_j^1\}$.
Since $\eta \geq |S|+2t + 2$, we can run this argument repeatedly to find distinct vertices $x_j^1, y_j^1, x_j^2, y_j^2,\dots, x_j^t, y_j^t$ outside $S\cup \{y_1,\dots, y_p\}\cup\{x_j\}$ such that $x_j^1y_j^1x_j^2y_j^2\dots x_j^t y_j^t$ forms a path $Q$ of length $2t$ in $G$ and that there is a member of $\F$ with leaf vector $\langle y_1,\dots, y_{j-1}, y_j^t, y_{j+1},\dots, y_p\rangle$. So, the first statement holds with $y'_j=y_j^t$. The proof for the second statement is almost identical,
except that we start with $F_0=F^*$ with leaf vector $\langle y_1,\dots, y_{j-1}, y^*_j, y_{j+1},\dots, y_p\rangle$ and we replace $S$ with $S\cup y_j$.
\end{proof}

\begin{lemma} \label{lem:path-linkage}
Let $k,h,p,j\geq 2$ be integers, where $k\geq j$.
Let $P$ be a vertex-labeled path of length $j$.
Let $\eta=2kh$. Let $\F$ be an $\eta$-robust family of copies of $P$ in $G$.
Let $F\in \F$. Let $F_0=F$ if $k\equiv j\pmod{2}$ and let $F_0$ be a subpath of $F$ of length $j-1$ if $k\not\equiv j\pmod{2}$. 
Then the endpoints of $F_0$ are $(k,h)$-linked in $G$.
\end{lemma}
\begin{proof}
First suppose $k\equiv j\pmod{2}$. Let $F\in \F$. Let $\langle x,y\rangle$ be the leaf vector of $F$. 
By our assumption, $k-j=2t$ for some nonnegative integer $t$. To prove that $x,y$ are $(k,h)$-linked in $G$, it suffices to show that for any set $W$ of at most $kh$ vertices not containing $x,y$, there exists an $x,y$-path of length $k$ in $G$ that avoids $W$.
Now, let such a $W$ be given.
Since $\F$ is $\eta$-robust and $\eta= 2kh\geq |W|+2t+2$, by Lemma \ref{lem:robust-structure}, there exists a member $F'\in \F$ with leaf vector $\langle x,y'\rangle$, where $y'\neq y$ and a $y,y'$-path $Q$ of length $2t$ in $G$ such that $V(Q)$ is disjoint from
$W$. Since $\F$ is $\eta$-robust, there are at least $\eta$ internally disjoint members with leaf vector $\langle x,y'\rangle$. Since $\eta\geq 2kh\geq |W|+|V(Q)|$, among them we can find a member $F''$ such that $V(F'')\setminus \{y'\}$ is disjoint from $W\cup V(Q)$. Now $F''\cup Q$ is a $x,y$-path of length $k$ in $G$ that avoids $W$. 

Next, suppose $k\not\equiv j \pmod{2}$ and that $F_0$ is a subpath of $F$ of length $j-1$. Without loss of generality, suppose $F_0=F-y$. Let $z$ be the unique neighbor of $y$ in $F$. We need to argue that $x,z$ are $(k,h)$-linked in $G$. As before, it suffices to argue that for any set $W$ of at most $kh$ vertices not containing $x,z$ there exists an $x,z$-path of length $k$ in $G$ that avoids $W$.
Since $\F$ is $\eta$-robust and $\eta\geq 2kh\geq |W|+2t+2$, by Lemma \ref{lem:robust-structure}, there exists a
member $F''\in \F$ with leaf vector $\langle x,y''\rangle$ and a $z,y''$-path $R$ of length $2t+1$ in $G$ such that $V(R)$ is disjoint from $W$.  Since $\F$ is $\eta$-robust, there are at least $\eta$ internally disjoint members with leaf vector $\langle x,y''\rangle$. Since $\eta\geq 2kh\geq |W|+|V(R)|$ we can find a member $F'''$ such that $V(F''')\setminus \{y''\}$ is disjoint from $W\cup V(R)$. Now $F'''\cup R$ is a $x,z$-path of length $k$ in $G$ that avoids $W$.
\end{proof}
The following lemma was given in \cite{Janzer-bipartite}. For completeness we include a proof using our framework.

\begin{lemma}\label{lem:spider-linkage}
Let $k,s,t\geq 2$ be integers.
Let $T$ be a vertex labeled spider with length vector $\langle \ell_1,\dots, \ell_s\rangle$,
where for each $i\in [s]$, $1\leq \ell_i\leq k$ and $\ell_i=1$ for at most one $i$.
Let $S_{s,k}$ denote the spider with $s$ legs of length $k$.
Let $G$ be a graph.
Let $\eta=2kst$. Let $\F$ be an $\eta$-robust family of copies of $T$ in $G$. Let $F\in \F$.
Let $\langle y_1,\dots, y_s\rangle$ be the leaf vector of $F$. For each $i\in [s]$, let $x_i=y_i$
if $\ell_i\equiv k \pmod{2}$ and let $x_i$ the unique neighbor of $y_i$ in $F$ if $\ell_i\not \equiv k\pmod{2}$.  Then $\langle x_1,\dots, x_s\rangle$ is $(S_{s,k},t)$-linked in $G$. In particular, $G$ contains a copy of $K_{s,t}^{(k)}$.
\end{lemma}
\begin{proof}
To prove that $\langle x_1,\dots, x_s\rangle$ is $(S_{s,k},t)$-linked in $G$, as in the proof of Lemma \ref{lem:path-linkage}, it suffices to prove that given any set $W$ of at most $kst$ vertices in $G$ not containing $x_1,\dots, x_s$, there exists a copy of $S_{s,k}$ with leaf vector $\langle x_1,\dots, x_s\rangle$ that is disjoint from $W$.
By our assumption, if $k-\ell_i$ is even, then $x_i=y_i$ and if $k-\ell_i$ is odd, then 
$x_i$ is the unique neighbor of $y_i$ in $F$. Since $\ell_i=1$ for at most one $i$,
$x_1,\dots, x_s$ are distinct vertices. By applying Lemma \ref{lem:robust-structure} $s$ times, we can find vertices $y'_1,\dots, y'_s\notin W$ and vertex disjoint paths $Q_1,\dots, Q_s$, where for each $i\in [s]$, $Q_i$ is a $x_i,y'_i$-path of length $k-\ell_i$ avoiding $W$ and  $\langle y'_1,\dots, y'_s\rangle$ is the leaf vector of some member of $\F$. We are able to do so, since $\eta\geq |W|+ks+2$.
Since $\F$ is $\eta$-robust, there exist at least $\eta$ internally disjoint copies of $T$ with 
leaf vector $\langle y'_1,\dots, y'_s\rangle$. Since
$\eta\geq |W|+ks$, among them we can find a member $F'\in \F$ such that $F'\cup (\bigcup_{i=1}^s Q_i)$ forms a copy of $S_{s,k}$ with leaf vector $\langle x_1,\dots, x_s\rangle$ that avoids $W$. Hence, $\langle x_1,\dots, x_s\rangle$ is $(S_{s,k},t)$-linked in $G$.

Finally, since $\langle x_1,\dots, x_s\rangle$ is $(S_{s,k},t)$-linked, there exist $t$ internally disjoint copies of $S_{s,k}$ with leaf vector $\langle x_1,\dots, x_s\rangle$. Their union forms a copy of $K_{s,t}^{(k)}$ in $G$.
\end{proof}

We recall the well-known asymmetric version of K\H{o}v\'ari-S\'os-Tur\'an theorem
(see \cite{JP} Lemma 1 for instance). 
\begin{lemma} \label{lem:KST}
Let $s,t\geq 2$ be integers. There exists a constant $c_{s,t}$ such that
for any $K_{s,t}$-free bipartite graph $G$ with part sizes $m,n$, $e(G)\leq c_{s,t}(mn^{1-1/s}+m+n)$.
\end{lemma}
The following lemma uses averaging arguments that originated in \cite{CJL}, except that here we pre-process the admissible paths of length $j$ to $(k,h)$-linked paths (of length $j$ or $j-1$) (Lemma \ref{lem:path-linkage}) and then apply the averaging to $(k,h)$-linked paths, whereas Conlon, Janzer, Lee \cite{CJL} applied the averaging to admissible paths of length $j$ and then applied additional arguments to build longer paths to link pairs.

\begin{lemma} \label{lem:bounds-on-paths}
Let $k,r,s,t$ be integers, where $k,s,t,\geq 2, r\geq 1$. Let $\eta\geq krst$.
Let $0<\ve<1, \mu\geq 1$ be positive reals. There exists a positive constant $D_0 = D_0(k,s,t,\ve, \mu, \eta)$ such that the following holds.
Let $G$ be an $\mu$-almost-biregular $K_{s,t}^{(k)}(r)$-free bipartite graph with an ordered partition $(M,N)$ and minimum degree at least $D_0$. 
Let $2\leq j\leq k$ be an integer. Let $P$ be a vertex-labeled path of length $j$.
 Then 
 \begin{enumerate}
 \item 
If $j$ is odd, then the number of $\eta$-heavy $\eta$-admissible copies of $P$ in $G$ is at most $4\mu \ve e(G)d_M^{\floor{\frac{j-1}{2}}} d_N^{\floor{\frac{j-1}{2}}}$. 
\item If $j$ is even then the
number of $\eta$-heavy $\eta$-admissible copies of $P$ in $G$ that start and end in $M$ is at most $2\mu \ve e(G) d_M^{\frac{j}{2}-1}d_N^{\frac{j}{2}}$ and the number of $\eta$-heavy $\eta$-admissible copies of $P$ in $G$ that start and end in $N$ is at most $2\mu \ve e(G) d_M^{\frac{j}{2}}d_N^{\frac{j}{2}-1}$, unless $j=2$ and $k$ is odd.
\end{enumerate}
\end{lemma}

\begin{proof} Let $C_1 = C_1(\ve, \mu, \eta, P_{k+1})$ as given in Lemma \ref{lem:robust-subfamily}
for the given $\ve, \mu, \eta$. We will choose $D_0$ to be at least $C_1$ and also sufficiently large to satisfy other conditions to be specified later.

{\bf Claim 1.}
The number of $(k,\eta)$-linked $2q$-paths in $G$ that start and end in $M$ is at most $\ve e(G) d_M^{q-1} d_N^q$. The number of $(k,\eta)$-linked $2q$-paths in $G$ that start and end in $N$ is at most $\ve e(G) d_M^q d_N^{q-1}$. The number of $(k,\eta)$-linked $(2q+1)$-paths in $G$ is at most $\ve e(G) d_M^q d_N^q$.

{\bf Proof of Claim 1.}
    The proofs of the three statements are essentially the same. For convenience, we will just prove the first statement, leaving the other two to the reader.
    Let $\F$ denote the family of $(k,\eta)$-linked $2q$-paths in $G$ that start and end in $M$. 
    For each $v\in N$, let $\F_v$ be the subfamily of $\F$ consisting of members whose second vertex is $v$. We claim that $|\F_v|\leq \ve d_M^{q-1}d_N^{q+1}$. Suppose for contradiction that $
    |\F_v|>\ve d_M^{q-1}d_N^{q+1}$.
    Let $X_v$ denote the set of first vertices on the members of $\F_v$ and $Y_v$ the set of last vertices on the members of $\F_v$.  
     Since $G$ is $\mu$-almost-biregular, we have $|X_v|\leq \mu d_N$ and $|Y_v|\leq \mu^{2q-1} d_M^{q-1}d_N^q$.
      Let $H_v$ be an auxiliary bipartite graph with parts $X_v, Y_v$ such that for each $x\in X_v, y\in Y_v$, we have $xy\in E(H_v)$ if and only if there is a member of $\F_v$ with first vertex $x$ and last vertex $y$. Since members of $\F_v$ are $\eta$-admissible and each has $v$ as the second vertex, for fixed $x,y$, there are at most $\eta^{(2q-1)^2}$ members of $\F_v$ with first vertex $x$ and last vertex $y$. Hence, 
    \[e(H_v)\geq |\F_v|/\eta^{(2q-1)^2}\geq (\ve/\eta^{(2q-1)^2})d_M^{q-1} d_N^{q+1}.\]
    Since $|X_v|\leq \mu d_N$, $|Y_v|\leq \mu^{2q-1} d_M^{q-1}d_N^q$ and $d_M,d_N\geq D_0$, by choosing $D_0$ to be sufficiently large, we have $e(H_v)> c_{s,t} (|X_v||Y_v|^{1-1/s}+|X_v|+|Y_v|)$, where $c_{s,t}$ is the constant given in Lemma \ref{lem:KST}. By Lemma \ref{lem:KST}, $H_v$ contains a copy $K$ of $K_{s,t}$. Without loss of generality, suppose say $K$ has parts $\{x_1,\dots, x_s\}\subseteq X$ and $\{y_1,\dots, y_t\}\subseteq Y$. For each $i\in [s],j\in [t]$, since $x_iy_j\in E(H_v)$, $x_i,y_j$ are $(k,\eta)$-linked in $G$. So there exist $\eta$ internally disjoint $x_i,y_j$-paths of length $k$ in $G$.
Since $\eta\geq krst$, this will allow us to greedily find a copy of $K_{s,t}^{(k)}(r)$ in $G$ with branching vertices $x_1,\dots, x_s, y_1,\dots, y_t$, contradicting $G$ being $K_{s,t}^{(k)}(r)$-free.  
    Hence, $|\F_v|\leq \ve d_M^{q-1}d_N^{q+1}$. Since this holds for each $v\in N$, we have
     $|\F|\leq \ve|N| d_M^{q-1}d_N^{q+1}\leq \ve e(G)d_M^{q-1}d_N^q$. \hfill $\Box$

Now, let $\langle v_1,v_2\rangle$ be the leaf vector of $P$.
Let $(A,B)$ be an ordered bipartition of $P$ so that $v_1\in A$.

{\bf Case 1.} $j$ is odd.

Let $j=2q+1$. Let $\F$ be the family of all $\eta$-heavy $\eta$-admissible copies of $P$ in $G$.
Suppose for contradiction that $|\F|>4\mu \ve e(G)d_M^q d_N^q$. 
Let $\F_1$ denote the subfamily of $\F$ consisting of members of $\F$ that map $v_1$ into $M$
and let $F_2=\F\setminus \F_1$. Suppose without loss of generality, $|\F_1|>2\mu \ve e(G) d_M^a d_N^q$.
Since $D_0\geq C_1$,
by Lemma \ref{lem:robust-subfamily}, $\F_1$ contains a $\eta$-robust subfamily $\F_1'$ with $|\F'_1|\geq (1/2)|\F_1|>\mu \ve e(G) d_M^qd_N^q$.
First, suppose $k$ is odd. By Lemma \ref{lem:path-linkage}, the endpoints of each member of $\F'_1$ are $(k,\eta)$-linked. By Claim 1, $|\F'_1|\leq \ve e(G)d_M^qd_N^q$, a contradiction. Hence, we may assume that $k$ is even. 
 Let $\D_1$ denote family of the $(P-v_2)$-projections of members of $\F'_1$. By Lemma \ref{lem:path-linkage},
each member of $\D_1$ is a $2q$-path that start and end in $M$ whose endpoints are $(k,\eta)$-linked. By Claim 1,
$|\D_1|\leq \ve e(G)d_M^{q-1}d_N^q$.
Since $G$ is $\mu$-almost-biregular, we have $|\F'_1|\leq |\D_1|(\mu d_M)\leq \ve \mu e(G)d_M^qd_N^q$, contradicting our earlier claim about
$|\F'_1|$.

{\bf Case 2.} $j$ is even.

Let $j=2q$. Let $\F$ denote the family of all $\eta$-heavy 
$\eta$-admissible copies of $P$ where $v_1,v_2$ are mapped into  $M$ and let $\G$ denote
the family of all $\eta$-heavy $\eta$-admissible copies of $P$ where $v_1,v_2$ are mapped into $N$. 
Suppose for contradiction that $|\F|> 2\mu \ve e(G) d_M^{q -1}d_N^{q}$.
By Lemma \ref{lem:robust-subfamily}, $\F$ contains a $\eta$-robust subfamily $\F'$ with $|\F'|\geq (1/2)|\F|$. If $k$ is even, then by Lemma \ref{lem:path-linkage}, each member of $\F'$ is a $2q$-path that starts and ends in $M$ whose two endpoints are $(k,\eta)$-linked. But this contradicts Claim 1. So, suppose $k$ is odd. As we exclude the $j=2$ and $k$ odd case,
$j\geq 4$.
Let $\D_1$ denote the family of the $(P-v_2)$-projections of members of $\F'$. By Lemma \ref{lem:path-linkage}, each member of $\D_1$ is a $(2q-1)$-path whose endpoints are $(k,\eta)$-linked.
By Claim 1, $|D_1|\leq \ve e(G) d_M^{q-1} d_N^{q-1}$ and hence 
$|\F'|\leq |\D_1|(\mu d_N) \leq \mu \ve e(G) d_M^{q-1}d_N^q$. 
Hence, $|\F|\leq 2|\F'|\leq 2\mu \ve e(G) d_M^{q-1}d_N^q$, contradicting our assumption.
By an almost identical argument, $|\G|\leq  2\mu\ve e(G) d_M^qd_N^{q-1}$. 
\end{proof}

Now, we are ready to prove the following key theorem, from which Theorem \ref{thm:main-kst} immediately follows.

\begin{theorem} \label{thm:upper-bound}
Let $\mu, k, r,s,t$ be  positive integers with $s, t, \mu \geq 2$. There exist positive constants $C$ and $C'$ depending on $\mu, k, r, s, t$ such that the following holds.
Let $G$ be a $\mu$-almost-biregular $K_{s,t}^{(2k)}(r)$-free bipartite graph with an ordered partition $(M,N)$ where $|M|\leq |N|$. We have $e(G)\leq C\left( |M|^{\frac{1}{2}+\frac{1}{2k}} |N|^{\frac{1}{2}}  +|M| + |N|\right)$.
Furthermore, if $G$ is $K_{s,t}^{(2k)}$-free, then $e(G)\leq C' \left( |M|^{\frac{1}{2}+\frac{1}{2k}-\frac{1}{2ks}} |N|^{\frac{1}{2}} + |M| + |N| \right)$.
\end{theorem}
\begin{proof}
Let $\ve=\frac{1}{4(8\mu)^{2ks}}$.
Let $\eta = 4 k^2 r s t $. 
Let $D_0$ be the constant returned by Lemma \ref{lem:bounds-on-paths}, with $2k$ in place of $k$.   Let $G$ be given as described.
We may assume that $d_M,d_N\geq \max\{4k, D_0\}$.  Let $P$ be a vertex-labeled path of length $2k$. 
Let $\F$ denote the family of copies of $P$ in $G$ that start and end in $M$. Since $d_M,d_N\geq 4k$, by a greedy argument, we have
\begin{equation}\label{eq:F-lower}
|\F|\geq |M|(d_M/2)^k(d_N/2)^k\geq \frac{1}{2^{2k}} |M|d_M^kd_N^k.
\end{equation}

Let $\F_1$ denote the subfamily of $\F$ consisting of members that contain an $\eta$-heavy (not necessarily proper) subpath. 
For each $F\in \F_1$, by Lemma \ref{lem:admissible-paths}, $F$ contains an $\eta$-heavy $\eta$-admissible $j$-path for some $2\leq j\leq 2k$.
For each $2\leq j\leq 2k$, by Lemma \ref{lem:bounds-on-paths} and the fact that $G$ is $\mu$-almost-biregular, the number of members of $\F_1$ that contain a $j$-subpath that is $\eta$-heavy $\eta$-admissible in $G$ is at most $4\mu^{2k-j+1}\ve|M|d_M^kd_N^k$.
Hence,
\begin{equation}\label{eq:F1-bound}
|\F_1|\leq\sum_{j=2}^{2k} 4\mu^{2k-j+1}\ve|M|d_M^kd_N^k\leq 8\mu^{2k - 1}\ve|M|d_M^kd_N^k.
\end{equation}
By definition, members of $\F\setminus \F_1$ are $\eta$-light in $G$. 
Since there are at most $|M|^2$ different leaf vectors, by \eqref{eq:F-lower} and \eqref{eq:F1-bound}, we have 
\[|M|^2\eta^{4k^2}\geq |\F\setminus \F_1|\geq (\frac{1}{2^{2k}}-8\mu^{2k - 1} \ve)|M|d_M^kd_N^k \geq \frac{|M|}{2^{2k+1}}d_M^{k}d_N^{k}
\geq  \frac{|M|}{2^{2k+1}\mu^{2k}}\frac{[e(G)]^{2k}}{|M|^k|N|^k}, \]
by our choice of $\ve$.  Solving for $e(G)$, we get 
$e(G)\leq 4 \mu \eta^{2k} |M|^{\frac{1}{2}+\frac{1}{2k}}|N|^{\frac{1}{2}}$.
Taking $C = \max\{4\mu \eta^{2k}, D_0\}$, this proves the first statement.

To prove the second statement, assume that $G$ is $K_{s,t}^{(2k)}$-free.  Let $C_1 = \max_{T \subseteq S_{s, 2k}} \{ C_1(\ve, \mu, \eta, T)\}$ returned by Theorem~\ref{lem:robust-subfamily}.
We may assume that $d_M,d_N\geq \max\{2ks, D_0, C_1\}$; otherwise the claim holds immediately. 
Let $T$ be a vertex-labeled spider with $s$ legs each having length $2k$. Let $(A,B)$ be an ordered bipartition of $T$ such that the center and the leaves of $T$ are in $A$.
Let $\T$ be the family of all copies of $T$ in $G$ that have $A$ embedded in $M$.  Since $d_M,d_N\geq 2ks$, we have
\[|\T|\geq |M|(d_M/2)^{ks}(d_N/2)^{ks})\geq \frac{1}{2^{2ks}} |M|d_M^{ks}d_N^{ks}.\]
Let $\T_1$ denote the subfamily of members of $\T$ that contain an $\eta$-heavy $j$-path in $G$ for some $2\leq j\leq 2k$. Since $G$ is $K_{s,t}^{(2k)}(r)$-free, as in the bound for $|\F_1|$, by Lemma \ref{lem:bounds-on-paths}
and the fact that $G$ is $\mu$-almost-biregular, 
\[|\T_1|\leq \sum_{j=2}^{2k} 4\mu^{2ks-j+1}\ve|M|d_M^{ks}d_N^{ks}\leq 8\mu^{2ks - 1}\ve|M|d_M^{ks}d_N^{ks}.\] 
Let $\T_2=\T\setminus \T_1$. Let $\T^*_2$ be the subfamily of $\T_2$ consisting of members
of $\T_2$ that are $\eta$-heavy in $\T_2$. By Lemma \ref{lem:admissible-spiders}, if $F\in\T^*_2$, then there exists an $s$-legged subspider $D$ of $T$ with at most one leg of length one such that $F[D]$ is $\eta$-heavy and $\eta$-admissible in $\T_2[D]$ . Let $\T_2^*[D]$ denote the subfamily of members of $\T_2[D]$ that are $\eta$-heavy and $\eta$-admissible in $\T_2[D]$. 
Let $a=|V(D)\cap A|, b=|V(D)\cap B|$.
If $|\T_2^*(D)|\geq 2\ve e(G) d_M^{b-1} d_N^{a-1}$, then by Lemma \ref{lem:robust-subfamily}, $\D$ contains a nonempty $\eta$-robust subfamily  and  by Lemma  \ref{lem:spider-linkage}, $G$ contains a copy of $K_{s,t}^{(2k)}$, contradicting $G$ being $K_{s,t}^{(k)}$-free. Hence, $|\T^*_2[D]|\leq 2\ve  e(G) d_M^{b-1} d_N^{a-1}$.
Since $G$ is $\mu$-almost-biregular, the number of members of $\T^*_2$  that contains a member of $\T^*_2[D]$ is thus at most
$\mu^{2ks-a-b} \ve |M| d_M^{ks}d_N^{ks}$. Clearly, there are fewer than $2^{2ks}$ different subspiders $D$ of $T$.
Hence, $|\T^*_2|\leq 2^{2ks}   \mu^{2ks-a-b} \ve |M| d_M^{ks}d_N^{ks}<(2\mu)^{2ks}\ve |M| d_M^{ks}d_N^{ks}$.

By our choice of $\ve$, the bounds on $|\T|, |\T_1|, |\T^*_2|$ above give 
\[|\T_2\setminus \T^*_2|=|(\T\setminus \T_1)\setminus \T^*_2)|\geq \frac{1}{2^{2ks+1}} |M| d_M^{ks}d_N^{ks}\geq 
\frac{1}{2^{2ks+1}\mu^{2ks}} |M|\frac{e(G)^{2ks}}{|M|^{ks}|N|^{ks}}.\]
.

But by defintion, members of $\T_2\setminus \T^*_2$ are $\eta$-light in $\T_2$. Since there are certainly at most $|M|^s$ leaf vectors of vertices in $M$, $|\T_2\setminus \T^*_2|\leq |M|^s \eta^{4k^2s^2}$. Combining the lower and upper bounds on $|\T_2\setminus \T^*_2|$, we get 
$e(G)\leq 4 \mu \eta^{2ks} |M|^{\frac{1}{2}+\frac{1}{2k}-\frac{1}{2ks}} |N|^{\frac{1}{2}}$, thus proving the statement for $C' = \max\{ 4 \mu \eta^{2ks}, D_0\}$.
\end{proof}

We are now ready to prove Theorem \ref{thm:main-kst}.

\begin{proof} [Proof of Theorem \ref{thm:main-kst}]
First, we prove that for all positive integers $k,m,n$ where $m\leq n$ and $k\geq 1$, $\ex(m,n, K_{s,t}^{(2k)}(r))=O(m^{\frac{1}{2}+\frac{1}{2k}} n^{\frac{1}{2}}+n\log m)$. Let $G$ be a bipartite graphs of part sizes $m$ and $n$ with $e(G) \geq \frac{C}{\lambda'} (m^{\frac{1}{2}+\frac{1}{2k}} n^{\frac{1}{2}}
+   n\log m)$, where $C$ is from Theorem~\ref{thm:upper-bound} and $\lambda'$ is from Corollary~\ref{cor:biregular} with $\alpha=\frac{1}{2}+\frac{1}{2k}, \beta=\frac{1}{2}$.
By Corollary \ref{cor:biregular}, $G$ contains a $16$-almost-biregular subgraph $G'$ with part sizes $m',n'$ such that $e(G') \geq C(m^{\frac{1}{2}+\frac{1}{2k}} n^{\frac{1}{2}}+   n' + m').$ Thus, by Theorem~\ref{thm:upper-bound}, $G'$ and thus $G$ has a copy of $K_{s,t}^{(2k)}(r)$, proving the first part of Theorem~\ref{thm:main-kst}. 

For the second part,we prove that for all positive integers $k,m,n$ where $m\leq n$ and $k\geq 1$, $\ex(m,n, K_{s,t}^{(2k)}(r))=O(m^{\frac{1}{2}+\frac{1}{2k} - \frac{1}{2ks}} n^{\frac{1}{2}}+n\log m)$. Let $G$ be a bipartite graphs with part sizes $m, n$ such that 
$e(G) \geq \frac{C'}{\lambda'}(m^{\frac{1}{2}+\frac{1}{2k} - \frac{1}{2k s}} n^{\frac{1}{2}}$ $+   n\log m)$, where $C'$ is from Theorem~\ref{thm:upper-bound} and $\lambda'$ is from Corollary~\ref{cor:biregular}, now with $\alpha=\frac{1}{2}+\frac{1}{2k}-\frac{1}{2ks}, \beta=\frac{1}{2}$. By Corollary \ref{cor:biregular}, $G$ contains a $16$-almost-biregular subgraph $G'$ with parts of size $m', n'$ such that $e(G) \geq C'[(m')^{\frac{1}{2}+\frac{1}{2k} - \frac{1}{2k s}} (n')^{\frac{1}{2}}+   m' + n']$. Thus, by Theorem~\ref{thm:upper-bound}, $G'$ contains a copy of $K_{s,t}^{(2k)}$, proving the second part of Theorem~\ref{thm:main-kst}. 
\end{proof}


\section{Proof of Theorem \ref{thm:main-kp}}
\label{sec:kp-regular}

Now, we address finding even subdivisions of $K_p$ in the bipartite setting.
Our overall strategy is based on that used by Janzer \cite{Janzer-longer} for the non-bipartite setting. However, there are some key differences noted as follows.
In the unbalanced bipartite setting,
 $2$-paths play a very delicate role. In the general setting, when an almost regular graph $G$ is $K_p^{(2k)}$-free, one can show that there are very few heavy $2$-paths. In the unbalanced bipartite setting, however, this can no longer be guaranteed. 
In an almost biregular bipartite graph $G$ with parts $M,N$ with $|M|\leq |N|$,
$2$-paths that start and end at $M$ (we will call these $(M,M)$-type) and $2$-paths that start and end at $N$ (we will call these $(N,N)$-type)  need to be handled differently. 
Our main novelty in the arguments in this section lies in our delicate handling of $2$-paths.
We will need two lemmas of contrasting nature. 
On the one hand, we prove that if $G$ is $K_{p,p}^{(k)}(r)$-free, where $k\geq 3$ is odd, 
then $G$ contains very few  $(N,N)$-type heavy $2$-paths.
On the other hand, we prove that if $G$ is $K_p^{(2k)}(r)$-free then it must contain a good number of $(M,M)$-type light $2$-paths. We then use this collection to build many $k$-paths in which every $(M,M)$-type $2$-path contained in them are light. From there we apply additional arguments to build $K_p^{(2k)}$.

\begin{lemma}\label{lem:2pathsodd}
Let $k, p, r$ be integers where $p \geq 2, r \geq 1$ and $k\geq 3$ is odd. Let $0 < \ve < 1$, $\mu \geq 1$ be positive reals. Let $\eta \geq 2^5 k^2p^2 r^2 \mu^3 \left( \frac{2}{\ve} \right)^{2p}$.
Then, there exists a positive constant $D_1$ depending on $k, p, \ve, \mu$ such that the following holds. 
Let $G$ be a $\mu$-almost-biregular $K_{p, p}^{(k)}(r)$-free bipartite graph with an ordered partition $(M,N)$ and minimum degree at least $D_1$.  
Let $P$ be a vertex-label path of length $2$.
Then the number of $\eta$-heavy copies of $P$ in $G$ that start and end in $N$ is at most $\ve |M|d_M^2$. 
\end{lemma}
\begin{proof} Let $(A,B)$ be the ordered partition of $P$ with $|A|=1, |B|=2$.
Let $D_1=C_1(\ve, \mu, \eta, P)$ returned by Lemma~\ref{lem:robust-subfamily}. Let $G$ be as given.
Let $\F$ denote the family of all $\eta$-heavy copies of $P$ in $G$ with $A$ embedded into $M$ and $B$ into $N$. Suppose for contradiction that $|\F|> 2\ve |M|d_M^2=2\ve e(G) d_M$.  By Lemma \ref{lem:robust-subfamily}, $\F$ contains a $\eta$-robust subfamily $\F'$ with $|\F'|\geq (1/2)|\F|
    \geq \ve |M|d_M^2$. Let $h=kp^2r$. Since $k-1$ is even and $\eta\geq 2(k-1)h$, by Lemma \ref{lem:path-linkage},
    the two endpoints of each member of $\F'$ are $(k-1,h)$-linked in $G$.
    By averaging there exists a vertex $v\in M$ that the subfamily $\F_v$ of $\F'$ consisting of members that map the middle vertex of $P$ to $v$ satisfies 
    $|\F_v|\geq |\F'|/|M|\geq \ve d_M^2$. Fix such a $v$. Let $H_v$ denote an auxiliary bipartite graph with parts $X,Y\subseteq N_G(v)$ such that $xy\in E(H_v)$ if and only if $xvy\in \F_v$.
    We have $e(H_v)=|\F_v|\geq \ve d_M^2$. 
    Since $v(H)\leq d_G(v)\leq \mu d_M$, $H$ contains a subgraph $H'$ with minimum degree at least $e(H)/v(H)\geq\frac{\ve}{\mu} d_M$.
    Let $X'=X\cap V(H'), Y'=Y\cap V(H')$. We first argue that we can find many vertices in $G$ that have many neighbors in $Y'$. Let $x$ be any vertex in $X'$. For each $y\in N_{H'}(x)$,
     since $xvy\in \F'$ and $\F'$ is $\eta$-robust, there are at least $\eta$ members of $\F'$ that have leaf vector $\langle x,y\rangle$. So there are at least $\frac{\eta \ve}{\mu} d_M$ members of $\F'$ that have $x$ as the first vertex. Let $\C=\{(w,y): xwy\in \F'\}$.
     Then $|\C|\geq \frac{\eta\ve}{\mu} d_M$. For each $w$, let $N_\C(w)=\{y\in Y': (w,y)\in \C\}$.
     Let $m=\ceil{2pr(\frac{2}{\ve})^p}$.
 
     Let $W$ denote the set of vertices $w$ satisfying $|N_\C(w)|\geq m^2$.  Since $|N_G(x)|\leq \mu d_N\leq \mu^2 d_M$,  the number of pairs in $\C$ incident to $W$ is at least $\frac{\eta\ve}{\mu}d_M-\mu^2 d_M m^2\geq \frac{\eta\ve}{2\mu}d_M$, by our definitions of $m$ and $\eta$. But for each $w\in W$, $|N_\C(w)|\leq\mu d_M$. Hence, $|W|\geq \frac{\eta \ve}{2\mu^2}\geq m$. Let $W_0$ be a set of $m$ vertices in $W$. Suppose $W_0=\{w_1,\dots, w_m\}$. Since $|N_\C(w_i)|\geq m^2$ for each $i\in [m]$, we can greedily find disjoint subsets $S_1,\dots, S_m$ of size $m$ such that for each $i\in [m]$ $S_i\subseteq N_\C(w_i)$. Let $S=\bigcup_{i=1}^m S_i$.
Let $U$ be a random $p$-subset of $X'$. Let $S'$ be the subset of vertices in $S$ that are adjacent to all of $U$ in $H'$.
For each $y\in S$, since $d_{H'}(y)\geq \ve|X'|$, the probability that $y$ is adjacent to all of $U$ in $H'$ is at least $\binom{|\ve|X'|}{p}/\binom{|X'|}{p}\geq (\ve/2)^p$,
where we used $\ve|X'|\geq 2p$. Hence, $\mathbb{E}(S')\geq (\ve/2)^pm^2$.
Fix such a $p$-set $U$ for which $|S'|\geq (\ve/2)^p m^2$.
For each $i\in [m]$, we call $i$ {\it good} 
if $|S'\cap S_i|\geq (\ve/2)^p m/2$; otherwise call it {\it bad}.
Let $\I_1$ denote the set of good $i$, and $\I_2$ the set of bad $i$.
By definition, $|\bigcup_{i\in \I_2} S'\cap S_i|<m(\ve/2)^p m/2$. So 
$|\bigcup_{i\in\I_1} (S'\cap S_i)|\geq (\ve/2)^pm^2/2$. But $|S'\cap S_i|\leq m$ for each $i$. So $|\I_1|\geq (\ve/2)^pm/2\geq p$. Without loss of generality, suppose $\I_1=[p]$.
For each $i\in [p]$, since $|S'\cap S_i|\geq (\ve/2)^p m/2\geq pr$, we can find $p$ disjoint subsets
$S_{i,1},\dots, S_{i,p}$ of $S'_i$ of $S'\cap S_i$, each of size $r$. For each $x_i\in U$ and $y\in \bigcup_{i=1}^p\bigcup_{j=1}^p S_{i,j}$, since $x_iy\in E(H_v)$, $x_i,y$ are $(k-1, h)$-linked in $G$.
Since $h=2kp^2r$, we can find internally disjoint $(k-1)$-paths joining each $x_i\in U$ to each $y\in \bigcup_{i=1}^p\bigcup_{j=p} S_{i,j}$, so that they also avoid $\{w_1,\dots, w_p\}$.
So, the union of these paths together with edges between $w_i$ to $\bigcup_{j=1}^p S_{i,j}$ for $i=1,\dots, p$ contains a copy of $K_{p,p}^{(k)}(r)$, a contradiction.
    \end{proof}

The method used in the above proof is not adaptable to $2$-paths that start and end in $M$ in the $k$ odd case. The issue is that $d_M$ can potentially be much larger than $d_N$.
For $2$-paths that start and end in $M$, we use a different strategy and instead prove that there should be a good amount of light $2$-paths.

\begin{lemma}\label{goodlightpaths}
Let $k, p, r$ be integers where $k \geq 2, p \geq 3$, and $r \geq 1$. Let $\mu \geq 1$ be a positive real. Let $\eta \geq p^4k^2r^2$. Then there exist positive constants $\gamma = \gamma(k, \mu, p, r)$, and $D_2 = D_2(k,\mu, p, r)$ such that the following hold. Let $G$ be $K_p^{(2k)}(r)$-free bipartite graph with an ordered bipartition $(M,N)$ and minimum degree at least $D_2$.
There is a collection of at least $\gamma |M|d_M^{\ceil{\frac{k}{2}}}d_N^{\floor{\frac{k}{2}}}$ many $k$-paths starting at $M$ such that every $2$-path starting at $M$ contained in it is $\eta$-light. 
\end{lemma}

\begin{proof}
Note that  if a path is $h$-light, it is $\eta$-light for every $\eta \geq h$. Thus, it suffices to construct $k$-paths starting at $M$ such that every $2$-path starting at $M$ is $h$-light for $h = p^4k^2r^2$. 
First, observe that $K_p^{(2k)}(r)$ is a bipartite graph in which vertices in one part have degree two and the other part has at most $p^2kr$ vertices. So it is a subgraph of $K^{(2)}_q$, where $q=p^2kr$.
Since $G$ is $K_p^{(2k)}(r)$-free, it must be $K_q^{(2)}$-free.
For each $v \in N$, form the graph $H_v$ with vertex set $N_G(v)$ and include $xy$ in $E(H_v)$ if and only if $xvy$ is an $\eta$-heavy copy of $P$ in $G$. If $H_v$ contains a copy $F$ of $K_q$ say on vertices $x_1,\dots, x_q$
then since each pair $\{x_i,x_j\}$ is joined by at least $h\geq q^2$ paths of length $2$ in $G$, we can greedily build a copy of $K_q^{(2)}$ in $G$ from $F$, contradicting our earlier claim. So $H_v$ is $K_q$-free. By Tur\'an's theorem, there must exist at least $(1/q)\binom{d_G(v)}{2}$ ordered pairs
$(x,y)$ such that $x,y\in N_G(v)$ and that $xvy$ is an $h$-light copy of $P$ in $G$. Hence, the number of $h$-light copies of $P$ in $G$ is at least
\[\sum_{v\in N} (1/q)\binom{d(v)}{2}\geq (1/q) |N| \binom{d_N}{2}\geq \frac{1}{4q}Nd_N^2\geq \frac{1}{4q\mu} |M|d_Md_N.\]

Now, form an auxiliary graph $\mathcal{M}$ on $M$ by joining two vertices if there is an $h$-light $2$-path between them. By definition of lightness and the previous count, this graph has at least $\frac{1}{8h q \mu}|M|d_Md_N$ edges. By successively deleting vertices of degree smaller than $\frac{1}{32 h q \mu} d_M d_N$, we can find a subgraph $\mathcal{M}'$ with minimum degree $\frac{1}{32  h q \mu } d_M d_N$ and at least  $\frac{1}{16h q \mu }|M|d_Md_N$ edges. Then to find many $\floor{\frac{k}{2}}$ paths in $\mathcal{M}'$, pick any edge and then greedily grow a $\floor{\frac{k}{2}}$-path from that edge. Since $\frac{1}{32 h q \mu}d_Md_N\geq 2 k$, we can find at least $(\frac{1}{64 h q \mu } d_Md_N)^{\floor{\frac{k}{2}}}|M|$ many  $\floor{\frac{k}{2}}$ paths in $\mathcal{M}'$. Since $G$ is $\mu$-almost-biregular, the total number of $2\floor{\frac{k}{2}}$-walks which are not paths because of a repeated $N$-vertex is no more than $\frac{1}{4}k^2\mu^{k - 1}|M|d_M^{\floor{\frac{k}{2}} - 1}d_N^{\floor{\frac{k}{2}}}$, since there are no more than $N$ many $N$-vertices, no more than $\floor{\frac{k}{2}}^2$ pairs of places on the path for an $N$-vertex to go, and no more than $\mu^{k - 1}d_M^{\floor{\frac{k}{2}} }d_N^{\floor{\frac{k}{2}} - 1}$ many such paths containing the given $N$-vertex in these two specific spots.   Since $d_M, d_N\geq D_2$, provided that $D_2$ is sufficiently large in terms of $\mu, h, q$, and $k$, we have that at least half of the paths we found in $\mathcal{M}'$ will translate to $2 \floor{\frac{k}{2}}$-paths in $G$. Let $\gamma = \frac{1}{4} (\frac{1}{64 h q \mu})^{\floor{\frac{k}{2}}}$. If $k$ is even, this is a collection of at least $2 \gamma |M| d_M^{\frac{k}{2}}d_N^{\frac{k}{2}}$ many $k$-paths starting at $M$ which contain only $h$-light $2$-paths starting at $M$.  If $k$ is odd, then using the minimum degree condition of $G$, we can then greedily extend this collection to a collection of at least $\gamma |M| d_M^{\ceil{\frac{k}{2}}}d_N^{\floor{\frac{k}{2}}}$ many $k$-paths starting at $M$ which contain only $h$-light $2$-paths starting at $M$. 
\end{proof}

Next, we need a useful structural lemma.

\begin{lemma}\label{lem:linkingpathskp}
    Let $G$ be a graph. Let $k,j$ be positive integers, where $k+1\leq j\leq 2k$. Let $\eta > 8p^2k^2r$.
    Let $P$ be a vertex-labeled path of length $j$ with leaf vector $\langle v_1,v_2\rangle$. 
    Let $\F$ be a $\eta$-robust family of copies of $P$ in $G$. 
    Then
    \begin{enumerate}
        \item For any $Q\in \F$, the endpoints of any subpath of $Q$ of length $2j-2k$ are $(2k,2p^2kr)$-linked in $G$. 
        \item Suppose $v$ is a vertex in $G$ such that there is a collection of $p$ paths $Q_1, Q_2, \dots Q_p$ of length $k$, each having $v$ as an endpoint such that for every $i,\ell\in [p]$, $i\neq \ell$, $Q_i\cup Q_\ell$ is a $2k$-path and contains a member of $\F$. Then $G$ contains a copy of $K_p^{(2k)}(r)$.
        \end{enumerate}
\end{lemma}
\begin{proof} Let $P_0$ denote any subpath of $P$ of length $2j-2k$. Let $h = 2p^2kr$. Let $u,v$ be the two endpoints of
$Q[P_0]$. Since $\F$ is $\eta$-robust and $\eta>2kh$, we can greedily find
members $Q_1,\dots, Q_h\in \F$ such that $Q_i[P_0]=Q[P_0]$ for each $i\in [h]$, and that they are pairwise vertex-disjoint outside $Q[P_0]$. (If $j=2k$, then $Q_i=Q$ for each $i\in [h]$.) For each $i\in [h]$, suppose
$Q_i$ has leaf vector $\langle x_i, y_i\rangle$. Since $\F$ is $\eta$-robust, for each $i\in [h]$,
there are $\eta$ internally disjoint members of $\F$ with leaf vector $\langle x_i,y_i\rangle$. Since $\eta>4kh$, we can
greedily find vertex disjoint members $F_1,\dots F_h$ of $\F$ such that for each $i\in [h]$, $F_i$ has leaf vector $\langle x_i,y_i\rangle$ and that the internal vertices of the $F_i$'s are outside
$\bigcup_{i=1}^h V(Q_i)$. Now $\{(Q_i\cup F_i)\setminus Q[P_0]: i\in [h]\}$ forms a collection of internally disjoint $u,v$-paths of length $2k$ in $G$. Hence, $u,v$ are $(2k,h)$-linked in $G$. 

Next, we prove the second statement. Let $Q_1,\dots, Q_p$ satisfy the given condition.
For each $i\in [p]$, let $v_i$ be the vertex of distance $j - k$ away from $v$ along $Q_i$. Fix any $i,\ell\in [p],i\neq \ell$. By our assumption, $Q_i\cup Q_\ell$ is a $2k$-path and contains a member $F$ of $\F$. Since $F$ has length $j$ and $Q_i$ has length $k$, $F$ must contain the portion of $Q_\ell$ from $v$ to $v_\ell$. Similarly, $F$ must also contain the portion of $Q_i$ from $v$ to $v_i$.
So $v_i,v_\ell$ are the endpoints of a subpath of $F$ of length $2j-2k$.
By statement one of the lemma, $v_i$ and $v_\ell$ are $(2k, 2p^2kr)$-linked in $G$.
We can thus greedily construct a copy of $K_p^{(2k)}(r)$ in $G$ with the $v_1, \dots v_p$ being the branching vertices.
\end{proof}

Now, we are ready to prove the second main theorem of the section.

\begin{theorem}\label{thm:kp-regular}
Let $\mu, k, r,p$ be  positive integers. There exist  positive constants $C''$ and $L''$ depending on $\mu, k, p, r$ such that the following holds.
Let $G$ be a $\mu$-almost-biregular $K_p^{(2k)}(r)$-free bipartite graph with an ordered partition $(M,N)$ where $|M|\leq |N|$. We have $e(G)\leq C''|M|^{\frac{1}{2}+\frac{1}{2k}} |N|^{\frac{1}{2}}+L''|N|$.
\end{theorem}
\begin{proof}  Throughhout this proof, $\eta$-heavy and $\eta$-admissible are all relative to $G$. We will specify $C'', L''$ later. Assume on the contrary that $e(G) \geq C'' |M|^{\frac{1}{2}+\frac{1}{2k}} |N|^{\frac{1}{2}} + L'' |N|$.  By choosing $L''$ sufficiently large, we may assume $\delta(G) \geq D_2(k, \mu, p, r)$ returned by Lemma~\ref{goodlightpaths}. Applying Lemma~\ref{goodlightpaths}, there exists a $\gamma = \gamma(k, \mu, p, r)$ such that $G$ contains a collection $\Pa$ of at least $\gamma |M|d_M^{\ceil{\frac{k}{2}}}d_N^{\floor{\frac{k}{2}}}$ labeled $k$-paths starting at $M$ such that each internal $2$-path starting at $M$ is $\eta$-light.    Recall that the Ramsey number $R_{s}(t)$ is the least number $q$ such that any $K_q$ edge-colored with $s$ colors will be guaranteed to have a monochromatic clique on $t$ vertices. Let 
    \begin{equation} \label{eq:epsilon-definition}
    \ve = \frac{\gamma^2}{2^{11}k^2\mu^{2k}R_{k + 1}(2p^2kr)}, \mbox{ and }
\eta = 2^5 k^2p^8 r^2 \mu^3 \left( \frac{2}{\ve} \right)^{2p^2}.
\end{equation}

    For each $k+ 1 \leq j \leq 2k$,
    let $\F_{j, 1}$ be the family of $\eta$-heavy and $\eta$-admissible labeled $j$-paths starting at $M$ and let $\F_{j, 2}$ be the family of $\eta$-heavy and $\eta$-admissible labeled $j$-paths starting $N$. By picking $L'' \geq 2 \mu C_1(\eta, \ve, \mu, P_{j+1})$ from Lemma~\ref{lem:robust-subfamily} for every $j$,  we may apply  Lemma~\ref{lem:robust-subfamily} to ensure that

{\bf Claim 1.} For each $j=k+1,\dots 2k$, $i=1,2$, $\F_{j, i}$  contains an $\eta$-robust subfamily $\F_{j, i}'$ such that $|\F_{j, i}' \setminus \F_{j, i}|$ is less than $\ve e(G)d_M^{\floor{\frac{j}{2}} - 1}d_N^{\ceil{\frac{j}{2}} - 1}$ if $i = 1$ and $\ve e(G)d_M^{\ceil{\frac{j}{2}} - 1}d_N^{\floor{\frac{j}{2}} - 1}$ if $i = 2$.

 Now,
Let $\F'=\bigcup_{j=k+1}^{2k}\bigcup_{i=1}^2 \F'_{j,i}$. Also, let $\F_N$ denote the family of all $\eta$-heavy $2$-paths in $G$ that start in $M$ (hence centered in $N$).

    For any $v \in V(G)$, we define an edge-colored graph $F_v$ with $V(F_v) = \{ P \in \Pa: v \text{ is the last vertex on } P\}$ as follows.  For any $P,Q\in V(F_v)$ with $P \cup Q$ a path, if the $2$-subpath of $P\cup Q$ centered at $v$ is in $\F_N$, we add $PQ$ to $E(F_v)$ and color it with color $2$; otherwise, if $P\cup Q$ contains an $\eta$-heavy $\eta$-admissible path in $\F'_{j, i}$ for some $j \geq k + 1$ and $i \in \{1, 2\}$ then we fix a $j$, add $PQ$ to $E(F_v)$ and color $PQ$ with color $j$. Observe if $k$ is odd, $F_v$ is nonempty only for $v \in N$, likewise if $k$ is even $F_v$ is nonempty only for $v \in M$. 
    

{\bf Claim 2.} $F_v$ has no clique on $R_{k +1}(2p^2kr)$ vertices. 

{\bf Proof of Claim 2.} Suppose for contradiction that $F_v$ has a clique on $R_{k + 1}(2p^2kr)$ vertices, then it would contain a monochromatic clique $K$ of size $2p^2kr$ in one of the colors from $\{2\}\cup\{k+1,\dots, 2k\}$. If $K$ is in color $2$,
then using the proof of Lemma~\ref{goodlightpaths} and the fact that $\eta\geq p^4k^2r^2$, we can find a copy of $K_p^{(2k)}{(r)}$ in $G$, a contradiction. So, suppose $K$ is colored $j$, for some $k+1\leq j\leq 2k$.
Then, by Lemma~\ref{lem:linkingpathskp} part (2) and $\eta > 4p^2k^2r$, $G$ would contain a copy of $K_p^{(2k)}(r)$, a contradiction. \hfill $\Box$

Let $\cL$ be the family of  $2k$-walks in $G$ of the form $P\cup Q$, where $P,Q\in \Pa$ such that $P\cup Q$ does not contain any member of $\F'$ or $\F_N$. Thus, we have that $|\cL|$ is exactly the sum of the number of nonedges in each $F_v$.  We give a lower bound on $|\cL|$ as follows. First, consider any $F_v$. If $|V(F_v)|\geq 2R_{k+1}(2p^2kr)$, then
    since $F_v$ has no clique on $R_{k + 1}(2p^2kr)$ vertices, it must have at least $\frac{1}{2 R_{k + 1}(2p^2kr) }\binom{|V(F_v)|}{2}$ nonedges by Tur\'an's Theorem. 
    Hence, for any $v$, $F_v$ contains at least $\frac{1}{2 R_{k + 1}(2p^2kr) }\binom{|V(F_v)|}{2}-2R^2_{k+1}(2p^2kr)$ nonedges. Hence,
    $|\cL|\geq \sum_{v\in V(G)} [\frac{1}{2 R_{k + 1}(2p^2kr) }\binom{|V(F_v)|}{2}-2R^2_{k+1}(2p^2kr)]$.  
    Note that $\sum_{v\in V(G)} |V(F_v)| \geq |\Pa|\geq \gamma|M|d_M^{\ceil{\frac{k}{2}}}d_N^{\floor{\frac{k}{2}}}$. If $k$ is odd, then $\sum_{v\in V(G)} |V(F_v)|=\sum_{v\in N} |V(F_v)|$ and
    \[ \sum_{v\in N} \binom{|V(F_v)|}{2} \geq \frac{1}{4|N|}\left( \sum_{v\in N} |V(F_v)| \right)^2 \geq \frac{\gamma^2 |M|^2}{4 |N|}d_M^{2\ceil{\frac{k}{2}}}d_N^{2\floor{\frac{k}{2}}} = \frac{\gamma^2 }{4 }|M|d_M^{k}d_N^{k}.\]
If $k$ is even, then $\sum_{v\in V(G)} |V(F_v)|=\sum_{v\in M} |V(F_v)|$ and
\[\sum_{v\in M} \binom{|V(F_v)|}{2} \geq \frac{1}{4|M|}\left( \sum_{v\in M} |V(F_v)| \right)^2 \geq \frac{\gamma^2 }{4}|M|d_M^{k}d_N^{k}.\]

Choose $L''$ to be sufficiently large so that $\frac{\gamma^2 }{8 R_{k + 1}(2p^2kr)}L''\geq 4R^2_{k+1}(2p^2kr)$. This ensures that
\[|\cL|\geq \frac{\gamma^2 }{8 R_{k + 1}(2p^2kr)}|M|d_M^{k}d_N^{k}-2R^2_{k+1}(2p^2kr)|N|\geq \frac{\gamma^2 }{16 R_{k + 1}(2p^2kr)}|M|d_M^{k}d_N^{k},\]
where we used $|M|d_M^kd_N^k\geq e(G)\geq L''|N|$. By the definitions of $\Pa, \F_N$, and $\cL$, each member of $\cL$ contains no $\eta$-heavy $2$-path that starts in $M$.

We also choose $L'' \geq 2 \mu D_1(k, p^2, \ve, \mu)$ returned by Lemma~\ref{lem:2pathsodd} and $L'' \geq 2 \mu D_0(k, p, p^2, \ve, \mu, \eta) $ returned by Lemma~\ref{lem:bounds-on-paths} when forbidding $K_{p, p^2}^{(k)}(r)$. Note that as $G$ is  $K_p^{(2k)}(r)$-free, $G$ is $K_{p, p^2}^{(k)}(r)$-free. 

In what follows, we will implicitly use the fact that for any $j$-path in $G$ that has $a$ vertices in $M$ and $b$ vertices in $N$, the number of ways to extend it to a $2k$-path that has $a'$ vertices in $M$ and $b'$ vertices in $N$ is at most $2k (\mu d_M)^{b'-b} (\mu d_N)^{a'-a}$.
    We observe the following: 
    \begin{enumerate}
        \item [(i)] The number of members of $\cL$ that are non-paths is no more than $4 k^2\mu^{2k - 1}\max\{ |M|d_M^kd_N^{k - 1}, |M|d_M^{k - 1}d_N^{k} \} =k^2\mu^{2k - 1}|M|d_M^kd_N^{k - 1} $. 
        \item [(ii)] The number of members in $\cL$ containing a $\eta$-heavy and $\eta$-admissible $2$-path starting at $N$ is at most $2 \ve k \mu^{2k - 2}|M|d_M^kd_N^k$ by Lemma~\ref{lem:2pathsodd} and the implicit fact mentioned earlier.
        \item [(iii)] The number of members in $\cL$ containing an $\eta$-heavy and $\eta$-admissible $j$-path for some $3 \leq j \leq k $ is at most $ 8 \ve k^2 \mu^{2k - 2}|M|d_M^kd_N^k $ by Lemma~\ref{lem:bounds-on-paths} and the implicit fact mentioned earlier.
        \item[(iv)] The number of members in $\cL$ containing a member of $ \F_{j, i} \setminus \F_{j, i}'$ for any $j\in [k+1, 2k], i\in [2]$ is at most $4\ve k^2 \mu^{k - 1} |M|d_M^kd_N^{k}$ by Claim 1 and the implicit fact mentioned earlier. 
    \end{enumerate}

Picking $L'' \geq \frac{\mu}{\ve}$, the number of members of $\cL$ counted in i-iv is no more than $32 \ve k^2 \mu^{2k - 2}|M| d_M^k d_N^k$. Thus by our choice of $\ve$ given in \eqref{eq:epsilon-definition}, we can remove these members of $\cL$ and be left with a subcollection $\cL'$ of $\cL$  of size $\frac{1}{32 R_{k + 1}(2p^2kr) }\gamma^2 |M|d_M^k d_N^k$ such that its members  contain no $\eta$-heavy and $\eta$-admissible path. Indeed, $\eta$-heavy $\eta$-admissible $2$-paths starting at $M$ are excluded since we established earlier that they are not contained in any member of $\cL$. The $\eta$-heavy $\eta$-admissible $2$-paths starting at $N$ are excluded since we removed those members of $\cL$ satisfying ii. Paths containing an $\eta$-heavy and $\eta$-admissible $j$-path for $3 \leq j \leq k$ are removed by item iii. 
By the definition of $\cL$, members of $\cL$ do not contain any member of $\F'$. Hence,
any member of $\cL$ that  contain an $\eta$-heavy and $\eta$-admissible $j$-path for $ k + 1\leq j \leq 2k$ must contain a member of $\bigcup_{j=1}^k \bigcup_{i=1}^2 (\F_{j,i}\setminus \F'_{j,i})$, which are removed by item iv.

    Since any $\eta$-heavy path contains an $\eta$-heavy and $\eta$-admissible path, we have that every member in $\cL'$ is $\eta$-light. But since $e(G)\geq C'' M^{\frac{1}{2}+\frac{1}{2k}}N^{\frac{1}{2}}$, we have that: 
    \begin{align*}
        |\cL'| &\geq \frac{1}{32 R_{k + 1}(2p^2kr) }\gamma^2 |M|d_M^k d_N^k\\
        &\geq \frac{1}{32 R_{k + 1}(2p^2kr) }\gamma^2  |M|^{1 - k}|N|^{-k} e(G)^{2k}\\
        &\geq( C'')^{2k}  \frac{\gamma^2}{32 R_{k + 1}(2p^2kr) } |M|^2.
    \end{align*}
    Selecting $C'' \geq \eta^{2k}(32 R_{k + 1}(2p^2kr))^{1/2k} \gamma^{-1/k}$, we have that $|\cL'| \geq \eta^{(2k)^2} |M|^2$, and thus by averaging there is a pair of vertices hosting at least $\eta^{(2k)^2}$ members of $\cL'$, a contradiction to every member of $\cL'$ being $\eta$-light.
\end{proof}

Now, we are ready to prove Theorem \ref{thm:main-kp}.

\begin{proof}[Proof of Theorem~\ref{thm:main-kp}]
Let $p, k, r$ be positive integers. Let $G$ be a bipartite graph with parts $(M, N)$ and $e(G) \geq \frac{A}{\lambda}(|M|^{\frac{1}{2} + \frac{1}{2k}}|N|^{\frac{1}{2}} + |N|\log|M|)$, where $A = \max\{ C'', L''\}$ with $C'', L''$ from Theorem~\ref{thm:kp-regular} and $\lambda$ is from Corollary \ref{cor:biregular} with $\alpha=\frac{1}{2}+\frac{1}{2k}, \beta=\frac{1}{2}$ and $|M| \leq |N|$. We will show $G$ has a copy of $K_p^{(2k)}(r)$. By Corollary \ref{cor:biregular}, $G$ contains a bipartite graph $G'$ with parts $M', N'$ such that $e(G') \geq A(|M'|^{\frac{1}{2}  + \frac{1}{2k}}|N'|^{\frac{1}{2}} + |M'| +|N'|)$. 
By Theorem~\ref{thm:kp-regular}, $G'$ and thus $G$ contains a copy of $K_p^{(2k)}(r)$.  Therefore, 
$\ex(m, n, K_{p}^{(2k)}(r) \leq O(m^{\frac{1}{2} + \frac{1}{2k}}n^{\frac{1}{2}} + n \log m)$. 
\end{proof}

\section{Lower bounds} \label{sec:lower}

In this section, we aim to show that for infinitely many pairs of $m, n$ the upper bound we established for $K_{s, t}^{(2k)}$ is asymptotically tight if $t$ is sufficiently large in terms of $s, k$, and also for fixed $s,t$ and sufficiently large $r$, the upper bound on for $K_{s,t}^{2k}(r)$ is asymptotically tight.

To do this, we make the following definition, adapted from Bukh and Conlon \cite{BC}. 
Let $T$ be a tree with an ordered bipartition $(A, B)$. Let $R$ be the set of leaves of  $T$. For a subset $S \subseteq V(T) \setminus R$, let $e(S) = |\{e \in E(T): e \cap S \neq \emptyset\}|$. 
Bukh and Conlon \cite{BC} defined  $T$ to be {\it balanced} if for all $S \subseteq V(T) \setminus R$, $$e(S) \geq \frac{|S|}{|V(T) \setminus R|} e(T). $$

Here, we define $T$ to be \textit{$\alpha$-balanced} relative to $(A,B)$ for some $\alpha \in \RR^+$, if for every $S \subseteq V(T) \setminus R$, 
\begin{equation} \label{eq:alpha-balance}
    e(S) \geq \frac{\alpha| S \cap A| + |S\cap B|}{\alpha |A \setminus R| +  |B \setminus R|}e(T).
\end{equation}
Observe that the function $g: \alpha \rightarrow \frac{\alpha| S \cap A| + |S\cap B|}{\alpha |A \setminus R| +  |B \setminus R|}e(T) $ is monotone in $\alpha$ on $(0,\infty)$. Thus, if $T$ is $\alpha_1$-balanced and $\alpha_2$-balanced, then $T$ is $\alpha$-balanced for every $\alpha \in [\alpha_1, \alpha_2]$.
Note that a balanced tree $T$ is  $1$-balanced relative to either of the two ordered bipartitions of $T$. However, a balanced tree $T$ may not be $\alpha$-balanced  relative to some bipartition $(A,B)$ for some $\alpha$. For instance, examine the path on four vertices with endpoints $u,v$ and add to each vertex  a leaf neighbor, let the resulting tree be $T$ and fix $(A,B)$ the bipartition of it in which $v\in B$. One can check $T$ is balanced but does not satisfy  \eqref{eq:alpha-balance} for $\alpha=\frac{2}{5}$ and $S = \{v\}$.  
 
The following is a useful lemma for $\alpha$-balanced trees.

\begin{lemma}\label{lem:e(H)}
Let $s$ be a positive integer and $\alpha$ a positive real. Let $T$ be a vertex-labeled tree with an ordered bipartition $(A,B)$ and $R$ the set of leaves. Suppose that $T$ is $\alpha$-balanced relative to $(A,B)$. Let $H\in \T^s(A,B)$. Then
 \[e(H) \geq \frac{ \alpha |A_H\setminus R| + |B_H \setminus R|}{\alpha |A \setminus R| + |B \setminus R| } e(T).\]
\end{lemma}
\begin{proof}
    We prove it by induction. The statement is trivially true for $s = 1$. Suppose now the the statement holds for $s - 1$. Then, $H$ is formed from some $H' \in \T^{s-1}$ by adding a copy $T_s$ of $T$. Let $A_s,B_s$ denote the images of $A,B$ in $T_i$, respectively.
    Let $S = V(T_s) \setminus V(H') = V(H) \setminus V(H')$. Then, $e(H) \geq e(S) + e(H')$. Since $T_s$ is $\alpha$-balanced relative to $(A_s, B_s)$ and hence relative to $(A_H,B_H)$, we have that  
    \[e(S) \geq \frac{\alpha| S \cap A_s| + |S\cap B_s|}{\alpha |A_s \setminus R| +  |B_s \setminus R|}e(T)= \frac{\alpha| S \cap A_H| + |S\cap B_H|}{\alpha |A \setminus R| +  |B \setminus R|}e(T).
    \] Combining this with the inductive bound on $e(H')$ and the fact
    that $|A_{H'}\setminus R|+|S\cap A_H|=|A_H\setminus R|$ and 
    $|B_{H'}\setminus R|+| S\cap B_H|=|B_H\setminus R|$, the result follows. 
\end{proof}

Before we prove the our lower bound we need the following lemmas from Bukh and Conlon \cite{BC}. Let $\mathbb{F}_q$ be the finite field with $q$ elements, and let $\bar{\mathbb{F}}_q$ be the algebraic closure. A variety of complexity at most $M$ is a set of zeroes of at most $M$ polynomials in at most $M$ variables, each with degree at most $M$. 
\begin{lemma}[Lemma 2.7 in \cite{BC}] \label{lem:complexity}
    Suppose $W$ and $D$ are varieties over $\bar{\mathbb{F}}_q$ of complexity at most $M$ which are defined over $\FF_q$. Then, one of the following holds for all $q$ sufficiently large in terms of $M$: 
    \begin{enumerate}
        \item $ |W(\FF_q) \setminus D(\FF_q)| \geq q / 2$, or 
        \item  $|W(\FF_q) \setminus D(\FF_q)| < c$, where $c = c_M$ depends only on $M$. 
    \end{enumerate}
\end{lemma}

\begin{lemma}[Lemma 2.3 in \cite{BC}]\label{lem:probfield}
    Suppose $q > \binom{t}{2}$ and $d \geq t - 1$. Then if $f$ is a random polynomial of degree at most $d$ in $j$ variables, and $x_1, \dots x_t$ are distinct points in $\mathbb{F}^j_q$, 
    $$\PP(f(x_i) = 0 \text{ for all } i = 1, \dots t) = \frac{1}{q^t}.$$ 
\end{lemma}

Given a vertex-labeled tree $T$ with an ordered bipartition $(A,B)$, let $\T^p(A,B)$ be the family of all possible bipartite graphs that are formed by taking the union of $p$ distinct labeled copies of $T$ sharing the same leaf vector $R$. Let $T^p(A,B)$ be the unique member of $\T^p$ in which the $p$ copies of $T$ sharing the same leaf vector $R$ are vertex disjoint outside $R$. For each $H \in \T^p(A,B)$, the common leaf vector $R$ of the trees joined to form $H$ will be called the {\it root vector} of $H$. Let $H\in \T^p(A,B)$. Suppose $T_1,\dots T_p$ are the $p$ copies of $T$ whose union is $H$. For each $i\in [p]$, let $A_i$ denote the image of $A$ in $T_i$ and $B_i$ the image of $B$ in $T_i$. It is easy to see that for $H$ to be bipartite, we must have $(\bigcup_{i=1}^p A_i)\cap (\bigcup_{i=1}^p B_i)=\emptyset$. We will define $A_H=\bigcup_{i=1}^p A_i$ and $B_H=\bigcup_{i=1}^p B_i$. Then $(A_H, B_H)$ is a bipartition of $H$. We will call $(A_H,B_H)$ the {\it ordered bipartition} of $H$ associated with $(A,B)$.
For the purpose of the next theorem, we need the notion of oriented trees. Given a tree $T$ together with an ordered bipartition $(A,B)$, we call the triple $(T,A,B)$ an {\it oriented tree}. 
Thus, $(T,A,B)$ and $(T,B,A)$ will be treated as two different oriented trees. 

\begin{lemma}  \label{lem:dense-construction}
    Let $\alpha$ be a positive rational number, and $\F$ a finite family of vertex labeled oriented trees $(T,A_T,B_T)$ such that $T$ is $\alpha$-balanced relative to $(A_T,B_T)$.
    Let \[\rho = \max \left \{\frac{\alpha |A_T \setminus R_T|+ |B_T \setminus R_T|}{e(T)}: (T,A_T,B_T)\in \F\right\}.\]
    Let $\ell$ be the smallest integer such that $\frac{1}{e(T)} \ell$ and $\frac{\alpha}{e(T)} \ell$ are integers for all $(T, A_T, B_T) \in \F$.  Then, there exist integers $C, q_0$ such that for every prime power $q \geq q_0$, there exists a bipartite graph $G$ with parts $M,N$ of sizes $q^{\alpha\ell}$ and $q^{ \ell}$, respectively, and at least $\frac{1}{2} q^{ \ell ( 1 + \alpha - \rho)}$ edges with the property that for each $(T,A_T,B_T)\in \F$ and each $H\in \T^C(A_T,B_T)$, $G$ does not contain a copy of $H$ with $A_H\subseteq M$ and $B_H\subseteq N$.
\end{lemma}
\begin{proof}
    Let $k = 2 \ell (1 + \alpha) \max\{v(T): (T,A_T, B_T)\in \F\}$ and $d = k \max\{v(T): (T,A_T,B_T)\in \F\}$.  Fix any prime power $q$ sufficiently large in terms of $\F, \ell$, and $\alpha$, and let $M$ be the set of $\alpha \ell$-tuples over $\FF_q$ and $N$ the set of $\ell$-tuples over $\FF_{q}$.  Note that $|M|=|N|^\alpha$. Let $\rho$ be defined as in the theorem. By our assumption on $\ell$, we have that $\rho \ell$ is a positive integer. Let $\Pa$ be a sequence of $\rho \ell$ independent random polynomials of degree at most $d$ in $(1 + \alpha) \ell$ variables, and define a bipartite graph $G_0$ with parts $M,N$ and edge set $E(G_0) =  \{(a, b)\in M\times N: \forall f \in \Pa, f(a, b) = 0\}$. By Lemma \ref{lem:probfield} and the fact that the polynomials are independently chosen, for each $(a,b)\in M\times N$, $\mathbb{P}((a,b)\in E(G_0))=q^{-\rho\ell}$.  Thus, $\mathbb{E}[|E(G_0)|| = q^{\ell ( 1 + \alpha - \rho)}$.

    Fix an oriented tree $(T, A_T, B_T)\in \F$ with root vector $R_T$. For a $|R_T|$-tuple $Z$ in $V(G_0)$, let $S(T,Z)$ be the set of copies of $T$  in $G_0$ with $R_T$ mapped to $Z$, $A_T$ mapped into $M$ and $B_T$ mapped into $N$. We wish to estimate $\mathbb{E}[S(T,Z)|]$. As in \cite{BC}, for technical reasons, it is convenient to estimate the $k$-th moment.    
    Let $\HH = \bigcup_{i = 1}^k \T^i(A_T,B_T)$. 
    For each $H \in \HH$, let $g(H,Z,k)$ be the number of ordered  collections of $k$ not necessarily distinct copies of $T$ in $M \times N$ with $R_T$ mapped to $Z$, $A_T$ mapped into $M$, $B_T$ mapped into $N$, whose union is a copy of $H$. 
    It is easy to see that $g(H,Z,k)\leq (kv(T))^{kv(T)} |M|^{|A_H \setminus R|}|N|^{|B_H \setminus R|}$.   
    Indeed, $(kv(T))^{k v(T)}$ is an upper bound for the number of labelings of some fixed set by a copy of $H$.
    Thus, by Lemma~\ref{lem:probfield} and the fact the polynomials in $\Pa$ are independently chosen, the expected size of 
    $|S(T,Z)|^k$ satisfies  
   \begin{align*}
       \EE[|S(T,Z)|^k] &= \sum_{H \in \HH} g(H,Z,k) q^{-(\ell e(H))\rho}\\
       &\leq (kv(T))^{kv(T)} \sum_{H \in \HH} |M|^{|A_H \setminus R_T|}|N|^{|B_H \setminus R_T|}q^{- \rho \ell e(H)}\\
       &\leq (kv(T))^{kv(T)} \sum_{H \in \HH} |N|^{ (\alpha |H_A \setminus R_T| + |H_B \setminus R_T| )} |N|^{- \rho e(H)},
   \end{align*}
    where we used $|M|=|N|^\alpha$ and $|N|=q^{\ell}$.    
    By our definition of $\rho$, we have 
$\rho \geq \frac{\alpha |A_T \setminus R_T| +  |B_T \setminus R_T|}{e(T)}$.
Since $T$ is $\alpha$-balanced relative to $(A_T,B_T)$, 
by Lemma \ref{lem:e(H)},
$e(H) \geq  \frac{ \alpha |H_A \setminus R_T| + |H_B \setminus R_T|}{\alpha |A_T \setminus R_T| + |B_T \setminus R_T| } e(T)$.
Therefore, the exponent on $|N|$ in any summand is no more than $0$.
 Therefore, each summand is at most one. As the number of summands is no more than $(kv(T))^{kv(T)}$, we have
 \[\EE[|S(T,Z)|^k] \leq (k v(T))^{2k v(T)}.\]

For convenience, let $p=|A_T\setminus R_T|$ and suppose $A_T\setminus R_T=\{x_1,\dots, x_p\}$.
Let $m=|B_T\setminus R_T|$ and suppose $B_T\setminus R_T=\{y_1,\dots, y_m\}$.
    Let $W(T,Z)$ be the variety defined as the set of $(u_1,\dots, u_p, v_1,\dots, v_m)$ with $u_i\in \FF_{q}^{\alpha \ell}$ for each $i\in [p]$ and $v_i\in \FF_{q}^{\ell}$ for each $i\in [m]$, satisfying
    \begin{enumerate}
        \item For every $f \in \Pa, i\in [p], j\in [m]$,  if $x_iy_j \in E(T)$, $f(u_i, v_j) = 0$.
        \item For every $f \in \Pa, i\in [p], z \in Z \cap V$, if $c \in R_T$ corresponding to $z$ is in an edge with $x_i \in A_T$, $f(x_i, z) = 0$. 
         \item For every $f \in \Pa, j\in [m]], z \in Z \cap U$, if $c \in R_T$ corresponding to $z$ is in an edge with $y_j \in B_T$, $f(z,y_j) = 0$.    
    \end{enumerate}

    Observe that $W(T,Z)$ is a variety in no more than $(1 + \alpha)\ell e(T)$ variables, defined by no more than $\rho \ell e(T)$ polynomials each of degree at most $d$. 

    Let $D(T,Z)$ be the subvariety of $W(T,Z)$ where for some $i,j\in [p], i\neq j$,  $u_i = u_j$, or  for some $i,j\in [m], i\neq j$, 
    $v_i = v_j$, or for some $z \in Z \cap M$ and $i \in [p]$, $z = u_i$, or for some $z \in Z \cap N$ and $j \in [m]$ $z = v_j$. Observe that $D(T,Z)$ is a variety in no more than $(1 + \alpha)\ell e(T)$ variables, defined by no more than $\rho \ell e(T) + v(T)^2$ polynomials each of degree at most $d$.  In particular, setting $M = \max\{(1 + \alpha)\ell e(T),  \rho \ell e(T) + v(T)^2, d\}$ and letting $C = c_M$ by Lemma~\ref{lem:complexity}, we have that either $|W(T,Z) \setminus D(T,Z)| < C$ or $|W(T,Z) \setminus D(T,Z)| \geq \frac{q}{2}$. But observe that by definition, $|W(T,Z) \setminus D(T,Z)| = |S(T,Z)|$. Therefore, 
    
    \[\PP(|S(T,Z)| \geq C ) = \PP(|S(T,Z)| \geq \frac{q}{2}) = \PP(|S(T,Z)|^k \geq 2^{-k} q^{k}) \leq \frac{(k v(T))^{2k v(T)}2^k}{q^{k}}.\]

    Let $X_T$ be the set of $|R_T|$-tuples $Z$ with $|S(T,Z)| \geq C$. Thus, since $k \geq 2 \ell(1 + \alpha) |R_T|$ and there certainly no more than $|R_T|^{|R_T|}q^{\ell(1 + \alpha)|R_T|}$ choices for possible $|R_T|$-tuples, we see that $\EE [|X_T|| \leq \frac{|R_T|^{|R_T|}(k v(T))^{2k v(T)}2^k}{q^{s\ell(1 + \alpha)}}$. For each $Z \in X_T$, we will remove all the edges adjacent to one vertex to form a subgraph $G$, which will remove all copies of $\T^C$ from $G_0$. Since $q$ is sufficiently large in terms of $\F, \alpha$ and $\ell$, we have $\sum_{T \in \F} |R_T|^{|R_T|}(k v(T))^{2k v(T)}2^k \leq \frac{1}{2} q^{\ell ( 1+ \alpha - \rho)}$.  Therefore, 
    \[\EE[E(G)]\geq \EE[|E(G_0)| - \max(q^{\ell}, q^{ \ell \alpha}) \sum_{(T,A_T,B_T) \in \F}|X_T] \geq \frac{1}{2} q^{ \ell ( 1 + \alpha - \rho)},\] as desired.
\end{proof}

We phrase this lemma in terms of families for the following reason. If we want to construct a bipartite graph $G$ with parts $U,V$ that 
does not contain any member $H$ of $\T^p(A,B)$ for a tree $T$ with an ordered bipartition $(A,B)$ we will need to forbid copies of $H$
with $A_H\subseteq U, B_H\subseteq V$ and copies of $H$ with $A_H\subseteq V, B_H\subseteq U$. To achieve that, we apply Lemma \ref{lem:dense-construction} with $\F=\T^p(A,B)\cup \T^p(B,A)$.
We next show that paths and spiders are sufficiently balanced. 

Note that in Lemma \ref{lem:dense-construction}, the construction is at least as large as the trivial lower bound $\Omega(\max\{ q^\ell, q^{\alpha \ell}\})$ only when $\rho$ is between $\alpha$ and $1$, which corresponds to $\frac{|B\setminus R|}{e(T)-|A\setminus R|}\leq \alpha\leq 
\frac{e(T)-|B\setminus R|}{|A\setminus R|}$ or equivalently, $\frac{|B\setminus R|}{|B\setminus R|+|R|-1}\leq \alpha\leq \frac{|A|+|R|-1}{|A\setminus R|}$.
For that reason, given a tree $T$ and an ordered bipartition $(A,B)$ of it, we define the {\it maximal interval} associated with $(A,B)$ as $I_T(A,B)= \left[ \frac{|B \setminus R|}{|B\setminus R|+|R|-1}, 
\frac{|A\setminus R|+|R|-1}{|A \setminus R|} \right]$. We omit the subscript $T$ when the context is clear.
Using $e(T)=|A|+|B|-1$, if we set $\alpha = \frac{|A \setminus R| +|R| - 1}{|A \setminus R|}$ in \eqref{eq:alpha-balance}, we get $e(S)\geq |S|+\frac{|R|-1}{|A\setminus R|} |S\cap A|$. If we divide the numerator and denominator of the right-hand side of \eqref{eq:alpha-balance} by $\alpha$ and set $\alpha=\frac{|B \setminus R|}{|B \setminus R| + |R| - 1}$, then we get $e(S)\geq |S|+\frac{|R|-1}{|B\setminus R|} |S\cap B|$. 
Hence, $T$ is $\alpha$-balanced relative to $(A,B)$ for every $\alpha\in I_T(A,B)$ if and only if for all $S\subseteq V(T)\setminus R$:
\begin{equation} \label{eq:balanced-conditions}
e(S)\geq |S|+\frac{|R|-1}{|A\setminus R|} |S\cap A| \mbox{ and }
e(S)\geq |S|+\frac{|R|-1}{|B\setminus R|} |S\cap B|.
\end{equation}
  
Recall that for a positive integer $m$, $P_m$ denotes the path on $m$ vertices.

\begin{lemma} \label{lem:path-balance}
Let  $P$ be a  path of length at least two and $(A,B)$ an ordered bipartition of $P$. 
Then $P$ is $\alpha$-balanced relative to $(A,B)$ for every $\alpha\in I_P(A,B)$.
Let $k$ be a positive integer. Then $P_{2k}$ is $\alpha$-balanced relative to any ordered partition of it for every $\alpha\in [\frac{k-1}{k}, \frac{k}{k-1}]$ and $P_{2k+1}$ is $\alpha$-balanced relative to both partitions of it for every $\alpha\in [\frac{k}{k+1}, \frac{k+1}{k}]$.
\end{lemma}
\begin{proof} Let $R$ be the set of leaves of $P$. Let $(A,B)$ be an ordered bipartition of $P$. By \eqref{eq:balanced-conditions}, it suffices to show for that any $S\subseteq V(T)\setminus R$,
$e(S)\geq |S|+\frac{1}{|A\setminus R|} |S\cap A|$ and 
$e(S)\geq |S|+\frac{1}{|B\setminus R|} |S\cap B|$.
By considering each component if needed, we may assume that $P[S]$ is connected.
Then $e(S)=|S|+1$. Since $|S\cap A|\leq |A\setminus R|$ and $|S\cap B|\leq |B\setminus R|$,
the inequalities hold.

Now, let $(A,B)$ denote any ordered partition of $P_{2k}$ and $(C,D)$ any ordered partition of $P_{2k+1}$. It is straightforward to verify that $I(A,B)\cap I(B,A)=[\frac{k-1}{k}, \frac{k}{k-1}]$ and $I(C,D)\cap I(D,C)=
   [\frac{k}{k+1}, \frac{k+1}{k}]$. The second part of the lemma then follows from the first part of the lemma.
\end{proof}
Recall that the theta graph $\Theta_{s,k}$
is the graph consisting of $s$ internally disjoint paths of length $k$ sharing the same endpoints.
Note that $\Theta_{s,k}=P_{k+1}^s$. 
The next two theorems deal with bipartite graphs with parts $M,N$ with $|M|\leq |N|$ that are
$H$-free. Since $|M|\leq |N|$, the value of $\alpha$ is upper bounded by $1$. 

\begin{theorem} \label{thm:theta-lower-odd}
Let $k\geq2$ be an integer. Let $\alpha\in [\frac{k-1}{k}, 1]$ be a rational number. Let $\ell$ be a smallest
integer such that $\frac{1}{2k-1}\ell$ and $\frac{\alpha}{2k-1}\ell$ are integers. 
There exist  integers $s_0, q_0$ such
that for every integer $s\geq s_0$ and prime power $q\geq q_0$ there exists a bipartite graph $G$ with parts $M,N$, where
$|M|=q^{\alpha\ell}, |N|=q^\ell$ such that $G$ is $\Theta_{s,2k-1}$-free and 
$e(G) \geq \frac{1}{2} |M|^{\frac{k}{2k - 1}}|N|^{\frac{k }{2k - 1}} .$
\end{theorem}
\begin{proof}
Note that $P_{2k}$ is symmetric relative to either bipartition. Let $(A,B)$ be either bipartition of it.
We apply Lemma \ref{lem:dense-construction} with $\F=\{(P_{2k},A,B)\}$ 
to find $G$ with parts $M,N$ and $e(G_0)]\geq \frac{1}{2} q^{\ell(1+\alpha-\rho)}$. Substituting in $\rho=\frac{\alpha(k-1)+(k-1)}{2k-1}$, $|M| =q^{\alpha \ell}, |N|=q^{\ell}$ and simplifying, we get $e(G)]\geq \frac{1}{2} |M|^{\frac{k}{2k - 1}}|N|^{\frac{k }{2k - 1}} $.
\end{proof}

Note that our interval for $\alpha$ is optimal since $|M|\leq |N|$ and when $\alpha=\frac{k-1}{k}$, the lower bound matches the trivial lower bound of $\Omega(|N|)$.

\begin{theorem} \label{thm:theta-lower-even}
Let $k\geq1$ be an integer. Let $\alpha\in [\frac{k}{k+1}, 1]$ be a rational number. Let $\ell$ be a smallest
integer such that $\frac{1}{2k}\ell$ and $\frac{\alpha}{2k}\ell$ are integers. 
There exist integers $s_0, q_0$ such
that for every integer $s\geq s_0$ and prime power $q\geq q_0$ there exists a bipartite graph $G$ with parts $M,N$, where
$|M|=q^{\alpha\ell}, |N|=q^\ell$ such that $G$ is $\Theta_{s,2k}$-free and 
$e(G)]\geq \frac{1}{2} |M|^{\frac{1}{2}+\frac{1}{2k}} |N|^{\frac{1}{2}}$.
\end{theorem}
\begin{proof}
Let $(A,B)$ be an ordered bipartition of $P_{2k+1}$. 
Apply Lemma \ref{lem:dense-construction} with $\F=\{ (P_{2k}, A,B), (P_{2k+1},B,A)\}$
to find $G$ with parts $M,N$ and $e(G_0)]\geq \frac{1}{2} q^{\ell(1+\alpha-\rho)}$. 
Here, $\rho=\max\{\frac{\alpha(k-1)+k}{2k}, \frac{\alpha k+k-1}{2k}\}=\frac{\alpha(k-1)+k}{2k}$, since $\alpha\leq 1$.
Substituting in $\rho=\frac{\alpha(k-1)+k}{2k}$, $|M| =q^{\alpha \ell}, |N|=q^{\ell}$ and simplifying, we get $e(G_0)]\geq \frac{1}{2} |M|^{\frac{1}{2}+\frac{1}{2k}} |N|^{\frac{1}{2}}$.
\end{proof}
Note that our interval for $\alpha$ is optimal since $|M|\leq |N|$ and when $\alpha=\frac{k}{k+1}$, the lower bound matches the trivial lower bound of $\Omega(|N|)$. Since $K_{s,t}^{(2k)}(r)$ contains $\Theta_{r,2k}$,  Theorem \ref{thm:theta-lower-even} shows that the upper bound in Theorem \ref{thm:main-kst} for $K_{s,t}^{(2k)}(r)$ and the upper bound in Theorem \ref{thm:main-kp} for $K_p^{(2k)}(r)$ are asymptotically tight when $r$ is sufficiently large. In \cite{exponent}, it is shown that a $\Theta_{s,p}$-free bipartite graph with part sizes $m,n$, where $m\leq n$,
has at most $O((mn)^{\frac{1}{2}+\frac{1}{2p}}+m+n)$ edges if $k$ is odd and at most $O(m^{\frac{1}{2}+\frac{1}{p}}n^{\frac{1}{2}}+m+n)$ edges if $k$ is even.
Theorem \ref{thm:theta-lower-odd} and Theorem \ref{thm:theta-lower-even} show that these upper bounds are asymptotically tight for sufficient large $s$.

\begin{lemma}\label{lem:spider-balance}
Let $T$ be a spider of $s$ legs, each with length $k$. Let $(A,B)$ be an ordered bipartition of $T$ such that $A$ contains all the leaves of $T$.
Then $T$ is $\alpha$-balanced relative to $(A,B)$ for every $\alpha\in I_T(A,B)$
and $T$ is $\alpha$-balanced relative to $(B,A)$ for every $\alpha\in I_T(B,A)$.
\end{lemma}
\begin{proof}
Let $R$ denote the set of leaves of $T$. Let $S\subseteq V(T)\setminus R$. It suffices to verify \eqref{eq:balanced-conditions}, which is equivalent to
\begin{equation} \label{eq:spider-balance}
e(S)\geq |S|+(s-1)\cdot \max\left\{ \frac{|S\cap A|}{|A\setminus R|} , \frac{|S\cap B|}{|B\setminus R|} \right\}.
\end{equation}
As before, we may assume that $T[S]$ is connected.
We consider two cases. In the first case, assume that $S$ contains the central vertex $w$ of $T$. In this case, we have  $e(S)=|S|+s-1$.
Since $|S\cap A|\leq |A\setminus R|$ and $|S\cap B|\leq |B\setminus R|$, we see that \eqref{eq:spider-balance} holds.

In the second case, assume that $w\notin S$. Since we assume $T[S]$ is connected, $S$ is contained in one leg of $T$ and  $e(S) = |S| + 1$.
 If $k$ is odd, then $|S \cap A| ,  |S \cap B| \leq  \frac{k - 1}{2}$ and  $|B \setminus R| = \frac{k - 1}{2}s + 1$ and $|A \setminus R| = \frac{k -1 }{2}s$. 
 It is straightforward to check that the right-hand side of \eqref{eq:spider-balance} is less than $|S|+1$. Hence, \eqref{eq:spider-balance} holds.
  If $k$ is even, then $|S \cap A|\leq \frac{k}{2}-1, |S \cap B|  \leq \frac{k}{2}, |A \setminus R| = \frac{k -2 }{2}s + 1$, and  $|B \setminus R| = \frac{k}{2}s$.  
Again,  it is straightforward to check that the right-hand side of \eqref{eq:spider-balance} is less than $|S|+1$. Hence, \eqref{eq:spider-balance} holds.
\end{proof}

\begin{theorem} \label{thm:kst-lower-odd}
Let $k\geq2$ be an integer. Let $\alpha\in [\frac{ks+1}{ks+s},1]$ be a rational number. Let $\ell$ be a smallest
integer such that $\frac{1}{(2k+1)s}\ell$ and $\frac{\alpha}{(2k+1)s}\ell$ are integers. 
There exist  integers $t_0, q_0$ such
that for every integer $t\geq t_0$ and prime power $q\geq q_0$ there exists a bipartite graph $G$ with parts $M,N$, where
$|M|=q^{\alpha\ell}, |N|=q^\ell$ such that $G$ is $K_{s,t}^{(2k+1)}$-free and 
$e(G)\geq \frac{1}{2}|M|^{1-\frac{k}{2k+1}} |N|^{1-\frac{k}{2k+1}-\frac{1}{(2k+1)s}}$.
\end{theorem}
\begin{proof}
Let $T=K_{1,s}^{(2k+1)}$ and $(A,B)$ the ordered bipartition. It is easy to check that
$I_T(A,B)\cap I_T(B,A)\supseteq [\frac{ks+1}{ks+s},1]$. By Lemma \ref{lem:spider-balance},
$T$ is $\alpha$-balanced relative to $(A,B)$ and to $(B,A)$ for every $\alpha\in  [\frac{ks+1}{ks+s},1]$.
Applying Lemma \ref{lem:dense-construction} with $\F=\{ (K_{1,s}^{(2k+1)}, A,B),$
$ (K_{1,s}^{(2k+1)},B,A)\}$ and $\alpha\in  [\frac{ks+1}{ks+s},1]$, we
obtain a bipartite $G$ with parts $M,N$, where  $|M| =q^{\alpha \ell}, |N|=q^{\ell}$,
and $e(G)\geq \frac{1}{2} q^{\ell(1+\alpha-\rho)}$. 
Here, 
\[\rho=\max\{\frac{\alpha ks +ks+1}{(2k+1)s}, \frac{\alpha(ks+1)+ks}{(2k+1)s}\}=\frac{\alpha ks+ks+1}{(2k+1)s},\]
since $\alpha\leq 1$.
Substituting in $\rho$, $|M|, |N|$ and simplifying, we get 
$e(G)\geq \frac{1}{2}|M|^{1-\frac{k}{2k+1}} |N|^{1-\frac{k}{2k+1}-\frac{1}{(2k+1)s}}$.
\end{proof}
Note that our interval for $\alpha$ is optimal since $|M|\leq |N|$ and when $\alpha=\frac{ks+1}{ks+s}$, the lower bound matches the trivial lower bound of $\Omega(|N|)$.
\begin{theorem} \label{thm:kst-lower-even}
Let $k\geq1$ be an integer. Let $\alpha\in [\frac{ks}{ks+s-1},1]$ be a rational number. Let $\ell$ be a smallest
integer such that $\frac{1}{2ks}\ell$ and $\frac{\alpha}{2ks}\ell$ are integers. 
There exist  integers $t_0, q_0$ such
that for every integer $t\geq t_0$ and prime power $q\geq q_0$ there exists a bipartite graph $G$ with parts $M,N$, where
$|M|=q^{\alpha\ell}, |N|=q^\ell$ such that $G$ is $K_{s,t}^{2k}$-free and 
$e(G)\geq \frac{1}{2}|M|^{\frac{1}{2}+\frac{1}{2k}-\frac{1}{2ks}} |N|^{\frac{1}{2}}$.
\end{theorem}
\begin{proof}
Let $T=K_{1,s}^{(2k)}$ and $(A,B)$ the ordered bipartition. It is easy to check that
$I_T(A,B)\cap I_T(B,A)\supseteq [\frac{ks}{ks+s-1},1]$. By Lemma \ref{lem:spider-balance},
$T$ is $\alpha$-balanced relative to $(A,B)$ and to $(B,A)$ for every $\alpha\in  [\frac{ks}{ks+s-1},1]$.
Applying Lemma \ref{lem:dense-construction} with $\F=\{ (K_{1,s}^{(2k)}, A,B),$
$ (K_{1,s}^{(2k)},B,A)\}$ and $\alpha\in [\frac{ks}{ks+s-1},1]$, we obtain a bipartite
$G$ with parts $M,N$, where $|M| =q^{\alpha \ell}, |N|=q^{\ell}$, and $e(G)\geq \frac{1}{2} q^{\ell(1+\alpha-\rho)}$. 
Here, 
\[\rho=\max\{\frac{\alpha [(k-1)s+1] +ks}{2k)s}, \frac{\alpha ks +(k-1)s+1}{2ks}\}=\frac{\alpha [(k-1)s+1]+ks}{2ks},\]
since $\alpha\leq 1$.
Substituting in $\rho, |M|, |N|$, and simplifying, we get 
$e(G)\geq \frac{1}{2}|M|^{\frac{1}{2}+\frac{1}{2k}-\frac{1}{2ks}} |N|^{\frac{1}{2}}$.
\end{proof}
Note that our interval for $\alpha$ is optimal since $|M|\leq |N|$ and when $\alpha=\frac{ks}{ks+s-1}$, the lower bound matches the trivial lower bound of $\Omega(|N|)$.
Theorem \ref{thm:kst-lower-even} shows that the upper bound on $K_{s,t}^{(2k)}$ in Theorem \ref{thm:main-kst} is asymptotically tight. 

\section{Concluding Remarks} \label{sec:conclusion}

\subsection{Weak regularization and error term}


Motivated by applications to incidence geometry, K\"odm\"on, Li, and Zeng \cite{KLZ} defined the {\it linear threshold} of a bipartite graph $H$ to be 
the smallest real $\sigma$ satisfying for any $\ve>0$ there exists $L$ such that every bipartite graph $G$ with an ordered bipartition $(M,N)$ with $|N|\geq \ve |M|^\sigma$ and $e(G)\geq L|N|$ must contain a copy of $H$. 
They showed that the linear threshold for $K_{s,t}^{(2)}$ is at most $2-1/s$, which is tight in the sense that for any $\sigma<2-1/s$, for sufficiently large $t$ the linear threshold of $K_{s,t}^{(2)}$ is larger than $\sigma$. Using the following lemma on a weaker notion of almost bi-regularity, their arguments can be used to prove the more general statement that  $\ex(m,n,K_{s,t}^{(2)})\leq O(m^{1-\frac{1}{2s}}n^{\frac{1}{2}}+n)$ for all $m\leq n$.
So this gives a better error term for $\ex(m,n,K_{s,t}^{(2)})$ than Theorem \ref{thm:main-kst} in the case of $K_{s,t}^{(2)}$.

  \begin{proposition} [Weak biregularization]\label{prop:weak-regularity}
       Let $c,\alpha, \beta, \ve$ be positive reals where $0\leq \alpha,\beta\leq 1$ and $\alpha + \beta > 1$.
       There exist constants $\lambda=\lambda(\alpha,\beta,\ve)$ and $\mu=\mu(\alpha,\beta,\ve)$ such that for every $L' \geq 16^{\frac{2}{\ve}}$ there exists an $L$ such that the following holds. 
       Let $G$ be a bipartite graph with an ordered bipartition $(M, N)$ with $|M| \leq |N|$ such that $e(G) \geq  c|M|^{\alpha}|N|^{\beta}$ and $d(G)\geq L$.
       Then $G$ contains a subgraph $G'$ with parts $M'\subseteq M, N'\subseteq N$ such that $e(G')\geq \lambda c 
       |M'|^\alpha |N'|^\beta$, $d(G')\geq  L'$ and that for any two vertices $x,y$ on the larger side of $(M,N)$, $d(x)\leq \mu d(y)$, and for any two vertices $x,y$ on the smaller side of $(M,N)$, $d(x)\leq d(y)^{1+\ve}$.
   \end{proposition}

   \begin{proof}  
     Let $\lambda_0=  \frac{1}{2^{2+\frac{1}{\alpha}+\frac{1}{\beta}}}$.
   By Lemma \ref{lem:onesideregular}, $G$  contains a subgraph $G_0$ with parts $M_0\subseteq M, N_0\subseteq N$ such that $e(G_0)\geq \lambda_0 c|M_0|^\alpha|N_0|^\beta$, $G_0$ is half regular at the larger side of $(M_0,N_0)$, and $d(G_0)\geq \frac{d(G)}{8}$.  Without loss of generality, suppose $N_0$ is the larger side. (The case of $M_0$ being the larger side is essentially identical.) 
   
   Let $\gamma=\frac{1}{2}(\alpha+\beta-1)$. First, suppose that $\lambda_0 c|M_0|^\alpha |N_0|^\beta> |N_0|^{1+\gamma}$.
   Then $d(G_0)\geq \frac{|N_0|^\gamma}{\lambda_0 c}$. Let $\lambda_1 =\lambda(\alpha,\beta)$ be given as in Theorem \ref{thm:biregular} for the given $\alpha,\beta$. By Theorem \ref{thm:biregular}, $G_0$ contains a $16$-almost-biregular subgraph $G'$ with parts $M'\subseteq M_0, N'\subseteq N_0$ such that $e(G')\geq \lambda \lambda_0 c|M'|^\alpha |N'|^\beta$
  and $d(G')\geq \frac{d(G_0)}{128 \log |M_0|}\geq \frac{|N_0|^\gamma}{\lambda_0 c (128\log |M_0|)}\geq L'$, if $|N_0|$ is sufficiently large, which can be guaranteed by making $L$ sufficiently large as a function of $\alpha,\beta, c,L'$. Now, since $d(G')\geq 16^{\frac{2}{\ve}}$, we have $\delta(G')\geq 16^{\frac{1}{\ve}}$, and hence $d(y)\leq 16 d(x)$ implies $d(y)\leq d(x)^{1+\ve}$. So the claim holds
  with $\lambda=\lambda_0 \lambda_1$ and $\mu=16$.

  Hence, we may assume that $\lambda_0 c|M_0|^\alpha|N_0|^\beta \leq |N_0|^{1+\gamma}$. 
  Recall that $d(G_0)\geq \frac{d(G)}{8}\geq \frac{L}{8}$, which implies $e(G_0)\geq \frac{L}{16}|N_0|$. This implies that $|M_0|\leq \left(\frac{1}{\lambda_0c}\right)^{\frac{1}{\alpha}}|N_0|^{\frac{1+\gamma-\beta}{\alpha}}$. 
   Let $d_{M_0}=\frac{e(G_0)}{|M_0|}$ and $d_{N_0}=\frac{e(G_0)}{|N_0|}$. By our assumption, $G_0$ is $d_{N_0}$-half regular at $N_0$. Also, $d_{M_0}\geq \frac{e(G_0)}{|M_0|}\geq \frac{L|N_0|}{16|M_0|}\geq \frac{1}{16}L(\lambda_0 c)^{\frac{1}{\alpha}}|N_0|^{1 - \frac{\alpha + 1 - \beta}{2\alpha }} \geq 4 |N_0|^{\frac{\alpha + \beta -  1 }{2\alpha }}$ by choosing $L$ large enough. Let $G_1$ be obtained from $G_0$ by iteratively deleting vertices in $M_0$ whose vertex degrees become less than 
   $\frac{1}{4}d_{M_0}$ and vertices in $N_0$ whose vertex degrees become less than $\frac{1}{4} d_{N_0}$. Then $e(G_1)\geq \frac{1}{2} e(G_0 )$. Furthermore, each vertex in $M_1:=V(G_1)\cap M$ has degree at least $\frac{1}{4} d_{M_0}\geq |N_0|^{\frac{\alpha + \beta - 1}{2 \alpha}}$ in $G_1$ and each vertex in $N_1:=V(G_1)\cap N$ has degree at least $\frac{1}{4} d_{N_0}$ in $G_1$. Let $D=|N_0|^{\frac{\alpha + \beta - 1}{2 \alpha}}$ and $p=\ceil{\log_{1 + \frac{\ve}{2}}(\frac{2\alpha}{\alpha + \beta - 1})}$. For each $i\in [p]$, let $A_i=\{x\in M_1: D^{(1+\frac{\ve}{2})^{i-1}}\leq  d_{G_1}(x) < D^{(1+\frac{\ve}{2})^i}\}$.
   Then $A_1,\dots A_p$ paritition $M_1$. By the pigeonhole principle, for some $i\in [p]$, the subgraph $G_2$ of $G_1$ induced by $A_i\cup N_1$ has at least $\frac{1}{p} e(G_1)\geq \frac{1}{2p}e(G_0)$ edges. Starting with $G_2$, by iteratively deleting any vertex in $A_i$ whose degree becomes less than $\frac{1}{4} \frac{e(G_2)}{|A_i|}$ and any vertex in $N_1$ whose degree becomes less than $\frac{1}{4} \frac{e(G_2)}{|N_1|}$, we obtain a subgraph $G'$ with parts $M'\subseteq A_i$ and $N'\subseteq N_1$ such that $e(G')\geq \frac{1}{2}e(G_2)\geq 
   \frac{1}{4p} e(G_0)\geq \frac{\lambda_0}{4p} c |M_0|^\alpha |N_0|^\beta\geq \frac{\lambda_0}{4p} c |M'|^\alpha |N'|^\beta$.
   Also, by our deletion rule, $\delta(G') \geq \frac{1}{8}d(G_2)\geq \frac{1}{16p} d(G_0)\geq \frac{1}{128p}L$, and so by choosing $L$ large enough, it will be larger than $L'$. Since $L' \geq 16^{\frac{2}{\ve}}$,  for any $x,y\in M'$, $d_{G'}(y)\leq 4d_{G'}^{1+\frac{\ve}{2}}(x) \leq d_{G'}^{1+\ve}(x)$, and for any $x,y\in N'$, $d_{G'}(x)\leq 4p d_{G'} (y)$.
   So the claim holds with $\lambda=\frac{\lambda_0}{4p}$ and $\mu=4p=4 \ceil{\log_{1 + \frac{\ve}{2}}(\frac{2\alpha}{\alpha + \beta - 1})}$.
\end{proof}

For our proofs of Theorem~\ref{thm:main-kst} and Theorem~\ref{thm:main-kp} to go through for $k \geq 2$, we needed the stronger version of biregularity.  
One natural question is whether the 
error term of $O(n \log m)$ in Theorem \ref{thm:main-kst} and Theorem \ref{thm:main-kp} can be improved to $O(n)$, which, if true, would be best possible.

\subsection{Odd subdivisions of $K_{s,t}$}

Our method does not quite allow us to establish good bounds on $\ex(m,n,K_{s,t}^{(k)}(r))$ when $k$ is odd. The main obstacle is that we are only able to prove Lemma \ref{lem:2pathsodd} for $2$-paths that start and end in $N$. It does not appear that $2$-paths that start and end in $M$ can be effectively used to build $K_{s,t}^{(k)}$ for odd $k$ when $|M|$ is much smaller than $|N|$. It is nevertheless natural to ask 

\begin{question}
    Let $k\geq 3$ be an odd integer. Is it true that for all positive integers $m\leq n$,  $\ex(m,n,K_{s,t}^{(k)}(r))\leq O(m^{\frac{1}{2}+\frac{1}{2k}} n^{\frac{1}{2} + \frac{1}{2k}}+n\log m)$ and 
    $\ex(m,n,K_{s,t}^{(k)})\leq O(m^{\frac{1}{2}+\frac{1}{2k}} n^{\frac{1}{2} + \frac{1}{2k}-\frac{1}{ks}}+n\log m)$?   
\end{question}

As shown in Theorem~\ref{thm:theta-lower-odd} and Theorem~\ref{thm:kst-lower-odd}, such bounds would be tight up to a constant for infinitely many pairs $m, n$.

\end{document}